\documentclass{amsart}
\usepackage{amssymb}
\usepackage{graphicx}

\newtheorem{theorem}{Theorem}[section]
\newtheorem{lemma}[theorem]{Lemma}
\newtheorem{proposition}[theorem]{Proposition}
\newtheorem{corollary}[theorem]{Corollary}

\newtheorem{remar}[theorem]{Remark}

\theoremstyle{definition}

\newenvironment{remark}{\begin{remar}\rm}{\diams\end{remar}}
\newcommand{\diams}{\unskip\nobreak\hfil\penalty50%
\hskip1em\hbox{}\nobreak\hfil%
$\diamondsuit$\parfillskip=0pt\finalhyphendemerits=0}

\newcommand{\bfind}[1]{\index{#1}{\bf #1}}
\newcommand{\n}{\par\noindent}

\newcommand{\sn}{\par\smallskip\noindent}
\newcommand{\mn}{\par\medskip\noindent}
\newcommand{\bn}{\par\bigskip\noindent}
\newcommand{\pars}{\par\smallskip}
\newcommand{\parm}{\par\medskip}
\newcommand{\parb}{\par\bigskip}
\newcommand{\N}{\mathbb N}
\newcommand{\Q}{\mathbb Q}
\newcommand{\R}{\mathbb R}
\newcommand{\Z}{\mathbb Z}

\newcommand{\cO}{{\mathcal O}}
\newcommand{\cM}{{\mathcal M}}
\newcommand{\cF}{{\mathcal F}}
\newcommand{\cC}{{\mathcal C}}
\newcommand{\cU}{{\mathcal U}}

\newcommand{\mbb}[1]{\underline{#1}}
\newcommand{\sep}{^{\rm sep}}
\newcommand{\ac}{^{\rm ac}}
\newcommand{\rc}{^{\rm rc}}
\newcommand{\id}{\mbox{\rm id}}

\newcommand{\chara}{\mbox{\rm char}\,}
\newcommand{\trdeg}{\mbox{\rm trdeg}\,}
\newcommand{\res}{\mbox{\rm res}\,}
\newcommand{\supp}{\mbox{\rm supp}\,}
\newcommand{\Spec}{\mbox{\rm Spec}}
\newcommand{\MaxSpec}{\mbox{\rm MaxSpec}}
\newcommand{\TopM}{\mbox{\rm Top}\,M}

\begin{document}
\title[Places and real holomorphy rings]{Density of Composite Places in Function Fields and Applications to Real Holomorphy Rings}
\author[Eberhard Becker, Franz-Viktor and Katarzyna Kuhlmann]{Eberhard Becker, Franz-Viktor Kuhlmann and
Katarzyna Kuhlmann}
\address{Technische Universit\"at Dortmund, Fakult\"at f\"ur Mathematik, Vogelpothsweg 87, 44227 Dortmund, Germany}
\email{eberhard.becker@tu-dortmund.de}
\address{University of Szczecin, Institute of Mathematics, ul.~Wielkopolska 15, 70-451 Szczecin, Poland}
\email{fvk@usz.edu.pl}
\address{University of Szczecin, Institute of Mathematics, ul.~Wielkopolska 15, 70-451 Szczecin, Poland}
\email{Katarzyna.Kuhlmann@usz.edu.pl}

\date{25.\ 11.\ 2020}

\begin{abstract} 
Given an algebraic function field $F|K$ and a place $\wp$ on $K$, we
prove that the places that are composite with extensions of $\wp$
to finite extensions of $K$ lie dense in the space of all places of
$F$, in a strong sense. We apply the result to the case of $K=R$ any
real closed field and the fixed place on $R$ being its natural
(finest) real place. This leads to a new description
of the real holomorphy ring of $F$ which can be seen as an 
analogue to a certain refinement of Artin's solution
of Hilbert's 17th problem. We also determine the relation 
between the topological space $M(F)$ of all $\R$-places of $F$ (places with residue field
contained in $\R$), its subspace of all $\R$-places of $F$ that are composite with the
natural $\R$-place of $R$, and the topological space of all $R$-rational places. Further results
about these spaces as well as various classes of relative real holomorphy rings are proven. At
the conclusion of the paper the theory of real spectra of rings will be applied to interpret
basic concepts from that angle and to show
that the space $M(F)$ has only finitely many topological components.
\end{abstract}

\thanks{The second author is supported by Opus grant 2017/25/B/ST1/01815 from the National
Science Centre of Poland.\\
The authors would like to thank the referee for his careful reading and many useful suggestions.}
\maketitle
%
%
%
%
\section{Introduction}
\subsection{The main results on places of algebraic function fields}
\mbox{ }\sn
The Main Theorem of \cite{KP} showed the density (in a very strong sense) of certain types of places in
the space of all places of a function field of characteristic 0 (by ``function field'' we will always mean
an algebraic function field of transcendence degree at least 1). A modification of the Main
Theorem was then applied to various classes of holomorphy rings, including the real and the $p$-adic. In a
subsequent paper \cite{[K]}, the Main Theorem was generalized to arbitrary characteristic. The density of
several important sets of places was shown, such as prime divisors, as well as the Abhyankar places which
play a crucial role e.g.\ in \cite{[K--K1],[K8]}. While the paper \cite{KP} only considered the space
\[
S(F|K)\>=\> \{\xi \mbox{ place of $F$ }\mid \xi|_K = \id_K\}
\]
of all places of an algebraic function field $F|K$ that are trivial on $K$, the scope was 
widened in \cite{[K]} to the spaces
\[
S(F|K\,;\,\wp)\>=\>\{\xi\in S(F)\mid \xi|_K=\wp\}
\]
of places of $F$ that extend a fixed place $\wp$ of $K$. We note that every $S(F|K\,;\,\wp)$
is a subset of the \bfind{Zariski space} $S(F)$ of all places of $F$, and that 
$S(F|K)=S(F|K\,;\,\id_K)$.

However, one interesting subset of these spaces was entirely missed: the set consisting of those
places that factor over $S(F|K)$ (see below for the precise definition). In this paper we will
adapt the proofs of the density theorems from \cite{KP,[K]} so as to prove the density of this 
subset and show how this is used to obtain ample information on families of real holomorphy
rings and the topologies of various spaces of places into formally real fields.

\parm
In order to present our central theorems, we need some preparations. In contrast to 
the usage in \cite{[K--K1],[K],[K5],[K8],[K7]} we will treat places as usual functions and apply
them to elements from the left, that is, the image of $a$ under a
place $\xi$ will be denoted by $\xi(a)$. However, we will keep one convention: the residue field
of $F$ under $\xi$ will be denoted by $F\xi$. Further, the valuation, valuation ring and
valuation ideal associated with the place $\xi$ will be denoted by $v_{\xi}\,$, $\cO_{\xi}$ and 
$\cM_{\xi}\,$, respectively. The value group of $v_\xi$ on $F$ will be denoted by $v_{\xi} F$.

\pars
Now we have all definitions in place to state our first central theorem.
\begin{theorem}                             \label{MKPZ}
Take an arbitrary field $K$ with a place $\wp$, a function field $F$ over $K$,
a place $\xi\in S(F|K\,;\,\wp)$,
and nonzero elements $a_1,\ldots,a_m\in F$.
Choose $r\in\N$ such that $1\leq r\leq s=\trdeg F|K$ and an arbitrary ordering on $\Z^r$; denote by $\Gamma$ the
so obtained ordered abelian group. If $\trdeg F|K>1$ and $\wp$ is trivial while $\xi$ is not, then we assume in
addition that $\Gamma$ is the lexicographic product $\Gamma'\times\Z$, where $\Gamma'=\Z^{r-1}$ endowed with
an arbitrary ordering.


\pars
Then there is a place $\lambda\in S(F|K)$ and an extension $\wp'$ of $\wp$ from $K$ to $F\lambda$
such that, with $\xi':=\wp'\circ\lambda\in S(F|K\,;\,\wp)$,
\sn
(a)\ \ $F\lambda$ is a finite extension of $K$,
\sn
(b)\ \ $v_{\lambda}F\subseteq\Gamma$ with $(\Gamma:v_{\lambda}F)$ finite,
\sn
(c) \ if $a_i\in \cO_\xi\,$, then $\lambda(a_i)\in \cO_{\wp'}$ and $a_i\in \cO_{\xi'}\,$.
\sn
The following assertions can also be realized if, in case  
$\wp$ is trivial, we assume that $\xi$ is trivial: 
\sn
(d) \ if $a_i\in \cM_\xi\,$, then $\lambda(a_i)\in \cM_{\wp'}$ and $a_i\in \cM_{\xi'}\,$,
\sn
(e) \ if $\xi(a_i)\in K\wp$, then $\xi'(a_i) = \xi(a_i)$,
\sn
(f)\ \ $\lambda(a_i)\ne 0,\infty$ \ for $1\leq i \leq m$.

\mn
If $\wp$ is trivial while $\xi$ is not, then in addition to assertions (a), (b)
and (c), we can realize also (d) and (e), or alternatively, (f).
\end{theorem}

We present two additions to the above theorem which work under stronger assumptions.
\begin{proposition}                         \label{MKPZa1}
Assume that the setting is as in Theorem~\ref{MKPZ}, and in addition, that $F\xi=K\wp$. 
In case $\wp$ is trivial, also assume that $\trdeg F|K>1$. Then in addition to the results of 
Theorem~\ref{MKPZ} we can also obtain that $F\xi'|K\wp$ is a finite purely inseparable extension.
\end{proposition}

Recall that a field $K$ is \bfind{existentially closed in an extension field $F$} if
every existential sentence in the language of rings with parameters from $K$ which holds
in $F$ will also hold in $K$. For further explanations, see \cite[Section 1]{KP}. Similarly, a
valued field $K$ (or an ordered field $K$ with a valuation) is existentially closed in an
extension field $F$ with an ordering extending that of $K$ (or ordering and
valuation extending those of $K$, respectively)
if every existential sentence in the language of rings with a relation symbol for
a valuation (or with relation symbols for an ordering and a valuation, respectively) and with
parameters from $K$ which holds in $F$ will also hold in $K$.
\begin{proposition}                         \label{MKPZa2}
Assume that the setting is as in Theorem~\ref{MKPZ}. 
If $(K,\wp)$ is existentially closed in $(F,\xi)$, then in addition to the results of 
Theorem~\ref{MKPZ} we can obtain that $F\lambda=K$, $v_{\lambda}F=\Gamma$, $F\xi'=K\wp$ and 
$\wp'=\wp$.
\end{proposition}

If $<$ is an ordering on the field $F$, then we will say that $\xi$ (or its associated valuation
$v_\xi$) is \bfind{compatible with $<$} if $\cO_\xi$ is convex relative to this given ordering.
That a place $\xi$ on $F$ is compatible with some of the orderings on $F$ is equivalent to the
statement that $F\xi$ is a formally real field. This is one essential part of the 
Baer-Krull Theorem (see \cite[Theorem 10.1.10]{BCR}).

Here is the convention we will follow: places into formally real fields are called \bfind{real
places}; valuation rings with a formally real residue field, i.e., the valuation rings of real
places, are called \bfind{real valuation rings}; finally, places $\xi$ with $F\xi\subseteq \R$
are denoted as \textbf{$\R$-places}.

\pars
Now we have all definitions in place to state our second central theorem, which adapts 
Theorem~\ref{MKPZ} to the real case.
\begin{theorem}                             \label{MKPZr}
Assume that $K=R$ is a real closed field and $F$ is an ordered function field over $K$. Assume
further that $\wp$ is a place on $R$ compatible with its ordering and $\xi$ is a place on $F$
compatible with the ordering $<$ of $F$. Take elements $a_1,\ldots,a_m\in F$ and let $r\in\N$ 
and $\Gamma$ be as in Theorem~\ref{MKPZ}. 
Then there is a place $\lambda\in S(F|R)$  such that, with $\xi':=\wp\circ\lambda\in 
S(F|R\,;\,\wp)$,
\sn
(a)\ \ $F\lambda=R$ and $F\xi'=R\wp$,
\sn
(b)\ \ $v_{\lambda}F=\Gamma$,
\sn
(c) \ if $a_i\in \cO_\xi\,$, then $\lambda(a_i)\in \cO_{\wp}$ and $a_i\in \cO_{\xi'}\,$,
\sn
(c') \ if $a_i>0$ and $\xi(a_i)\ne 0,\infty$, then $\xi'(a_i) >0$.
\sn
The latter implies that if $\infty\ne\xi(a_i)>0$, then $\xi'(a_i) >0$.

\mn
The following assertions can also be realized if, in case  
$\wp$ is trivial, we assume that also $\xi$ is trivial:
\sn
(d) \ if $a_i\in \cM_\xi\,$, then $\lambda(a_i)\in \cM_{\wp}$ and $a_i\in \cM_{\xi'}\,$,
\sn
(e) \ if $\xi(a_i)\in R\wp$, then $\xi'(a_i) = \xi(a_i)$,
\sn
(f)\ \ $\lambda(a_i)\ne 0,\infty$ \ for $1\leq i \leq m$,
\sn
(g) \ if $a_i>0$, then $\lambda(a_i) >0$.

\mn
If $\wp$ is trivial while $\xi$ is not, then in addition to assertions (a), (b), (c) 
and (c'), we can realize also (d) and (e), or alternatively, (f) and (g).
\end{theorem}

\sn
\begin{remark}
In the case where $\wp$ is trivial while $\xi$ is not, assertion {\it (d)} is incompatible
with {\it (f)} and {\it (g)} because in this case, $\cO_{\xi'}=\cO_{\lambda}$. Hence if $0\ne a_i\in\cM_\xi$ and
$\lambda$ satisfies assertion {\it (d)}, then $a_i\in\cM_{\lambda}$, hence $\lambda(a_i)=0$ so that assertion
{\it (f)} is not satisfied by $\lambda$; if in addition $a_i>0$, then also assertion
{\it (g)} is not satisfied by $\lambda$.
\end{remark}

\begin{proposition}                             \label{MKPZa3}
Assume that the setting is as in one of the above theorems or propositions. Then there are
infinitely many nonequivalent places $\lambda$ and $\xi'$ which satisfy all assertions of the
respective theorem or proposition, except in the case where $\wp$ is trivial while $\xi$ is not, 
$\trdeg F|K=1$ and we wish that assertions {\it (d)} and {\it (e)} are satisfied.
\end{proposition}

\begin{remark}
If $\trdeg F|K=1$ and $\wp$ is trivial while $\xi$ is not, then $\xi$ itself satisfies assertions
{\it (a)}, {\it (b)}, {\it (c)}, {\it (d)}, {\it (e)} of the theorems, and also assertion 
{\it (c')} in the setting of Theorem~\ref{MKPZr}. This will be shown in the proofs of the two
theorems. But it may be the only such place. If for instance $\wp$ is the identity, $F=K(x)$, 
$\xi$ is the $x$-adic place and $a_1=x$, then $\xi(x)=0\in K=K\wp$ so that by assertion 
{\it (e)}, $\lambda(x)=\xi'(x)=\xi(x)=0$, whence $\lambda=\xi$.
\end{remark}

The above theorems and propositions will be proven in Section~\ref{sectprac}.
The proof of Theorem~\ref{MKPZr} uses the fact that a real closed field is
existentially closed in every formally real extension field; see parts 3) and 4) of Theorem~\ref{realexcl}. This
allows us to apply Proposition~\ref{MKPZa2} as well as
Theorem~\ref{exrp}, a main ingredient to our proofs which is
a generalization of Theorem~23 of \cite{[K]}.

\bn
%
%
\subsection{Applications to spaces of places of algebraic function fields}
\mbox{ }\sn
Given two places $\xi'$ and $\pi'$, we will say that $\xi'$ \bfind{factors over $\pi'$} (or in other words, is
\bfind{composite with $\pi'$}) if there is a place $\lambda$ such that $\xi'=\pi'\circ\lambda$. Theorem~\ref{MKPZ}
shows the strong density of the subset of $S(F|K\,;\,\wp)$ of all places $\xi'$ that factor over $\wp'$ for a
suitable finite extension $(K',\wp')$ of $(K,\wp)$. Let us describe one consequence of the strong
density. Every set $S(F|K\,;\,\wp)$ carries the \bfind{Zariski topology}, for which the basic open sets are the
sets of the form
\begin{equation}                            \label{ZT}
\{\xi\in S(F|K\,;\,\wp)\mid a_1\,,\ldots,\,a_k\in {\cO}_\xi\}\;,
\end{equation}
where $k\in\N\cup\{0\}$ and $a_1,\ldots,a_k\in F$. With this topology, $S(F|K\,;\,\wp)$ is a
spectral space (see \cite[Appendix]{[K]} for a proof, and \cite{[H]} for details on spectral
spaces); in particular, it is quasi-compact. Its associated \bfind{patch
topology} (or \bfind{constructible topology}) is the finer topology
whose basic open sets are the sets of the form
\begin{equation}                            \label{PT}
\{\xi\in S(F|K\,;\,\wp)\mid a_1\,,\ldots,\,a_k\in {\cO}_\xi\,;\,
a_{k+1}\,,\ldots,\,a_{k+\ell}\in {\cM}_\xi\}\;,
\end{equation}
where $k,\ell\in\N\cup\{0\}$ and $a_1,\ldots,a_{k+\ell} \in F$. With the patch topology, 
$S(F|K\,;\,\wp)$ is a totally disconnected compact Hausdorff space. 

Note that every set $S(F|K\,;\,\wp)$ contains nontrivial places since $\trdeg F|K>0$ 
by our general assumption. Therefore, from 
Theorem~\ref{MKPZ} (in particular assertions {\it (c)} and {\it (d)})
in connection with Proposition~\ref{MKPZa3}, we obtain:
\begin{corollary}                                    \label{corMKPZ}
Take a function field $F|K$ and a place $\wp$ on $K$. Then every nonempty open set in the 
Zariski topology of $S(F|K\,;\,\wp)$ contains infinitely many places that factor over $\wp'$ 
for a suitable finite extension $(K',\wp')$ of $(K,\wp)$. The same holds in the Zariski patch
topology, unless $\trdeg F|K=1$ and $\wp$ is trivial.
\end{corollary}

\parb
Now we turn to our applications of Theorem~\ref{MKPZr}. Let us give the definitions
necessary to deal with formally real function fields $F$ over real closed fields $R$.
By
\[
M(F)
\]
we will denote the set of all \bfind{$\R$-places} of $F$, that is, places $\xi$ of $F$ with
residue field $F\xi
\subseteq\R$. These are exactly (up to equivalence of places) the places associated with the natural valuations of
the orderings on $F$, where the natural valuations of an ordered field $(F,<)$ is the finest valuation compatible
with that ordering. In particular, every real closed field $R$ has a unique $\R$-place $\xi_R\,$, which we will
call its \bfind{natural $\R$-place}.

Instead of the set $S(F|R)$ of all places of $F$ that are trivial on $R$, we are rather 
interested in the set
\[
M(F|R)\>=\> \{\lambda\in S(F|R)\mid F\lambda =R\}\>.
\]
of \bfind{$R$-rational places}. The new object we study in this paper is the set
\[
M_R(F) \>:=\> \{\xi_R \circ\lambda \mid \lambda\in M(F|R)\}\>\subseteq\>S(F|R\,;\,\xi_R)
\]
of all $\R$-places of $F$ that factor over $\xi_R\,$. Theorem~\ref{MKPZr} implies
that for every $\R$-place $\xi$ of $F$ there is an
$\R$-place $\xi'$ of $F$ that factors over $\xi_R$ and is ``very close to $\xi$''.

\pars
\begin{remark}
Note that we usually do not identify equivalent real places. However, here any two equivalent
places in $M_R(F)$ are equal since their residue fields are equal to the archimedean real closed
field $R\xi_R\subseteq\R$ which does not allow any nontrivial order preserving embedding in $\R$.
Also in $M(F|R)$, by its definition as a subset of 
$S(F|R)$, equivalent places are equal. Hence $M(F|R)$ is in general smaller than the set of all
$R$-rational places of $F$, but every such place is equivalent to a place in $M(F|R)$. As we are
interested in the compositions of $R$-rational places with the natural $\R$-place $\xi_R$ of $R$, 
this constitutes no loss of information. Indeed, assume that $\lambda_1$ and $\lambda_2$ are
equivalent $R$-rational places, and write $\lambda_2=\sigma\circ\lambda_1$ for some isomorphism 
$\sigma$. As $\lambda_2$ is assumed to be $R$-rational, $\sigma$ must be an automorphism of $R$.
Since $R$ is real closed, it is also order preserving. As $\cO_R:=\cO_{\xi_R}$ is the convex hull
of $\Q$ in $R$ and $\Q$ is left elementwise fixed by $\sigma$, it follows that 
$\sigma\cO_R=\cO_R$. This implies that $\xi_R$ and $\xi_R\circ\sigma$ are equivalent, and with 
the same argument as before, we find that they are equal. Thus, $\xi_R\circ\lambda_1$ and 
$\xi_R\circ\lambda_2=\xi_R\circ\sigma\circ\lambda_1$ are equal.
\end{remark}


\parm
Theorem~\ref{MKPZr} is essential for the description of the relation between the sets $M(F|R)$,
$M_R(F)$ and $M(F)$. It will be applied to formally real function fields $F$ over
arbitrary real closed fields $R$ where we address the following issues.

\parm
The set $M(F)$, its subset $M_R(F)$ and the set $M(F|R)$ carry natural topologies. The topology of $M(F)$ as
described by Dubois in \cite{D} is compact and Hausdorff; we will denote it by $\TopM(F)$. It is a quotient
topology of the space of orderings with the Harrison topology. Its basic open sets are
\[
U(f_1,...,f_m)\>:=\> \{\xi \in M(F)\mid \xi(f_i)>0 \mbox{ for } 1\leq i\leq m\}
\]
where  $f_1,...,f_m$ lie in the \bfind{real holomorphy ring} $H(F)$ of the field $F$. 

The concept of the real holomorphy ring $H(L)$ of any formally real field $L$ is due to Dubois, 
cf.~loc.~cit.. It is defined to be the intersection of all real valuation rings of $L$; it is 
equal to the intersection of the
valuation rings of all $\R$-places of $L$, and $L$ is its field of fractions. 

The holomorphy ring $H(L)$ turns out to be a Pr\"ufer ring as the localizations at its prime
ideals are exactly the real valuation rings of $L$. In particular, a valuation ring of $L$ is 
a real valuation ring if and only if it contains $H(L)$. The basic theory of the real 
holomorphy ring is developed in \cite{Be,Sch1}; consult \cite{Fo} for the general theory of
Pr\"ufer rings. Note that if $R$ is a real closed field, then $H(R)=\cO_R\,$.

In this paper, we will also be dealing with various relative real holomorphy rings in $L$ which are defined as the intersections of some family of real valuation rings. These relative real holomorphy rings are overrings of the (absolute) real holomorphy ring $H(L)$. Hence, by the general theory of Pr\"ufer rings they are Pr\"ufer rings as well. 

When we speak of the topological space $M(F)$, we will always refer to $\TopM(F)$, and the subset $M_R(F)
\subseteq M(F)$ will always carry the subspace topology.
So far, the topological space $M_R(F)$ has not found any attention in the literature. Yet, for our present
study it is highly relevant. In particular, Proposition~\ref{dense} will
show that $M_R(F)$ is dense in $M(F)$, which is a very important fact.

\pars
In an analogous way, a topology $\TopM(F|R)$ on $M(F|R)$ will be introduced. It is then shown in
Theorem~\ref{topspaces} that the mapping
\begin{equation}                              \label{MFRtoMRF}
\iota_{F|R}: M(F|R)\rightarrow M(F), \lambda \mapsto \xi_R\circ\lambda
\end{equation}
is a topological embedding with image $M_R(F)$. In the same theorem, it is shown that all three
topological spaces have no isolated points.

Similar to the space $M_R(F)$, the space $M(F|R)$ has found little, if any, attention in real algebraic geometry.
It was passed by in favour of stronger topological spaces, see e.g.\ \cite{Sch}. In the concluding section of this
paper we re-address these three topological spaces by invoking the theory of real spectra of rings, a cornerstone
of modern real algebraic and semi-algebraic geometry, see \cite{BCR}. As a surprising application we derive that
the space $M(F)$, where $F$ is a formally real function field over any real closed field, admits only finitely many
connected components.


\parm
So far, various authors have already studied the following \bfind{relative real holomorphy rings}
\[
H(F|R) \>:=\> \{a\in F\mid \xi(a)\ne\infty \mbox{ for all real places } \xi\in S(F|R)\}
\]
for function fields $F$ over real closed fields $R$, and its extensions
\[
H(F|R) D
\]
(the smallest subring of $F$ containing $H(F|R)$ and $D$), where $D$ is a finitely generated $R$-algebra inside
$F$, see \cite{BCR, BS, BuKu, JR, KS, K1, K2, KP, Sch, Sch1, Sch2}. Model theory or algebraic geometry or a
combination of both theories have been used. Common to all of these approaches is that they use the fact that $F$
admits many smooth models (projective or real complete affine
ones), which in turn allows to study the behaviour of the elements in $F$ as functions on the set $M(F|R)$.

In the case of a non-archimedean real closed base field $R$, this relationship seems to get lost once one turns
to the absolute real holomorphy ring $H(F)$ in place of $H(F|R)$. However, using the set $M_R(F)$
of all $\R$-places of $F$ that factor over the natural $\R$-place of $R$, we are able to prove representations for
$H(F)$ and related rings that still retain the geometric flavour; see Section~\ref{sectHr}.

Theorem~\ref{MKPZr} allows much wider application to all composite places which factor over places in $M(F|R)$. It
is this strength that allows to broadly extend previous results on relative real holomorphy rings. In fact, we
can include the class of rings  $H(F)D$ where $D$ is a general finitely generated ring extension over any real
valuation ring $B$ of the base field $R$.

\mn
%
%
\section{Proof of Theorems~\ref{MKPZ} and~\ref{MKPZr} and the related propositions}           \label{sectprac}
We will need the following fact, which has been shown in \cite[Theorem~1.1]{[BK]}:
\begin{proposition}                         \label{trdegmi}
Let $L|K$ be an extension of finite transcendence degree, and $v_\xi$ a nontrivial valuation on
$L$ with associated place $\xi$. If $v_\xi L/v_\xi K$ is not a torsion group or $L\xi|K\xi$ is
transcendental, then $(L,v_\xi)$ admits an immediate extension of infinite transcendence degree.
\end{proposition}

\pars
The proofs of our central theorems and propositions are adaptations of the proof of the Main Theorem in \cite{KP},
but instead of the Ax-Kochen-Ershov Theorem used there we will have to use other transfer
principles. Namely, we will need analogues for
algebraically closed fields, algebraically closed fields with valuation, ordered real closed
fields, and ordered real closed fields with compatible valuation. We also include a result on
divisible ordered abelian groups that is analogous to the one on algebraically closed (ordered) 
fields.
\begin{theorem}                                        \label{realexcl}
1) \ In the language of rings, an algebraically closed field $K$ is existentially closed in every
extension field $F$.
\sn
2) \ In the language of rings with a relation symbol for a valuation, an algebraically closed
nontrivially valued field $K$ is existentially closed in every valued extension field $F$.
\sn
3) \ In the language of rings with a relation symbol for an ordering, a real closed field $R$ is
existentially closed in every ordered extension field~$F$.
\sn
4) \ In the language of rings with relation symbols for an ordering and a valuation, a real 
closed field $R$ with nontrivial compatible valuation is existentially closed in every ordered
extension field $F$ equipped with a compatible valuation which extends the valuation of $R$.
\sn
5) \ In the language of groups with a relation symbol for an ordering, a nontrivial divisible
ordered abelian group $\Gamma$ is existentially closed in every ordered abelian group extension
$\Delta$. 
\end{theorem}
\begin{proof}
1): \ Take an algebraic closure $F\ac$ of $F$. By the model completeness of the theory of
algebraically closed valued fields (see \cite{RO}), $F\ac$ is an elementary extension of $K$ in
the language of rings. Every existential
sentence in this language with parameters from $K$ that holds in $F$
also holds in $F\ac$, and by what we just have stated, it then also holds in $K$. This proves 
that $K$ is existentially closed in $F$ in this language.
\sn
2): \ Take an algebraic closure $F\ac$ of $F$ together with some extension of the valuation. By
Abraham Robinson's theorem on the model completeness of the theory of algebraically closed
nontrivially valued fields, see \cite[Theorem 3.4.21]{RO}, $F\ac$ is an
elementary extension of $K$ in the language of rings with a relation symbol for a valuation. The 
remainder of the argument is as in the proof of part 1).
\sn
3): \ Take a real closure $F\rc$ of $F$ together with the corresponding extension of the 
ordering. Then the ordering on $F\rc$ extends the unique ordering of the real closed field $R$. 
By \cite[Theorem 4.5.1]{DP}, $F\rc$ is an elementary extension of $R$ in the language of rings
with a relation symbol for an ordering. Now our assertion follows as in the proof of part 1), 
with $F\rc$ and $R$ in place of $F\ac$ and $K$, respectively.
\sn
4): \ Take a real closure $F\rc$ of $F$ together with the corresponding extensions of the 
ordering and the compatible valuation of $F$. Again, the ordering on $F\rc$ extends the unique
ordering of the real closed field $R$. As the
compatible valuation on $F\rc$ extends the one of $F$, which in turn extends the one of $R$, it also extends the
one of $R$. By \cite[Corollary 4.5.4]{DP} and the fact that the ordering is definable in a real closed field in
the language of rings, $F\rc$ is an elementary extension of $R$ in the language of rings with relation symbols
for an ordering and a valuation. Now our assertion follows as in the proof of part 3).
\sn
5): \ Take any divisible hull $\tilde\Delta$ of $\Delta$. By Abraham Robinson's theorem on the
model completeness of the theory of nontrivial divisible ordered abelian groups, see 
\cite[Theorem 3.1.13]{RO}, $\tilde\Delta$ is an elementary extension of $\Gamma$
in the language of groups with a relation symbol for an ordering. The 
remainder of the argument is as in the proof of part 1).
\end{proof}

\pars
Further, we will need a generalization of Theorem~23 of \cite{[K]}.
\begin{theorem}                          \label{exrp}
Let $F|K$ be an algebraic function field and choose $\Gamma$ as in Theorem~\ref{MKPZ}. Take any 
nonzero elements $a_1,\ldots,a_m\in F$. Then there are infinitely many (nonequivalent) places 
$\lambda\in S(F|K)$ such that $F\lambda|K$ is finite, $v_\lambda F\subseteq\Gamma$ with 
$(\Gamma:v_\lambda F)$ finite, and $\lambda(a_i)\ne 0,\infty$ for $1\leq i\leq m$.
\pars
If in addition $K$ is existentially closed in $F$, then these places can be chosen to be 
$K$-rational with $v_\lambda F=\Gamma$.
\end{theorem}
\begin{proof}
We adapt the proof of the lemma on p.~190 of \cite{KP}. In some algebraic closure $F\ac$ of $F$ 
we find an algebraic closure $K_0$ of $K$ and let $F':=K_0.F$ be the field compositum of $K_0$ 
and $F$ inside of $F\ac$. By part 1) of
Theorem~\ref{realexcl}, $K_0$ is existentially closed in $F'$.

Since $K_0|K$ is algebraic, $\trdeg F'|K_0=
\trdeg F|K=s$. The extension $F'|K_0$ is separable and finitely generated, so we can
pick in $F'$ a separating transcendence basis $t_1,\ldots,t_s$ together with an element
$y$ separable algebraic over $K_0(t_1,\ldots,t_s)$ such that
$F'=K_0(t_1,\ldots,t_s,y)$. Take $f\in K_0[t_1,\ldots,t_s,Y]$ to be an
irreducible polynomial of $y$ over $K_0[t_1,\ldots,t_s]$. We write $\mbb{t}=(t_1,\ldots,t_s)$ and
\begin{equation}                                     \label{a_i}
a_i \>=\> \frac{g_i(\mbb{t},y)}{h_i(\mbb{t})} \mbox{\ \ \ for\ \ \ } 1\leq i \leq m\;,
\end{equation}
where $g_i$ and $h_i$ are polynomials over $K_0$, with $h_i(\mbb{t})\ne 0$. Since the elements
$t_1,\ldots,t_s,y$ satisfy
\begin{equation}                                     \label{satsent}
f(\mbb{t},y)=0\>,\; \frac{\displaystyle\partial f}{\displaystyle\partial Y}
(\mbb{t},y)\not= 0\;\mbox{ \ and \ }\; h_i(\mbb{t})\ne 0 \>\mbox{ for } 1\leq i\leq m
\end{equation}
in $F'$, we infer from $K_0$ being existentially closed in $F'$ that there
are $t'_1,\ldots,t'_s,y'$ in $K_0$ such that
\[
f(\mbb{t'},y')=0\>,\; \frac{\displaystyle\partial f}{\displaystyle\partial Y}
(\mbb{t'},y')\not= 0\;\mbox{ \ and \ }\; h_i(\mbb{t'})\ne 0 \>\mbox{ for } 1\leq i\leq m\>.
\]
Now let $K_1$ be the subfield of $K_0$ which is generated over $K$ by the
following elements:\sn
$\bullet \ \ t'_1,\ldots,t'_s, y'\,$,\n
$\bullet \ \ $ the coefficients of $f$, $g_i$ and $h_i$ for $1\leq i\leq m$.
\sn
We note that $K_1$ is a finite extension of $K$. We will now construct an extension $K_4$
of $K_1$ with $K_1$-rational place $\lambda_4$, which will contain an isomorphic copy of $K_1.F$. 
The construction will be done in such a way that the place $\lambda$ induced on $F$ through the
resulting embedding of $F$ in $K_4$ and the place  $\lambda_4$ will satisfy the assertions of 
our theorem.

\pars
We write $\Gamma=\bigoplus_{1\leq i\leq r} \Z\alpha_i$ with $\alpha_i>0$. We adjoin $r$ many
algebraically independent elements $x_1,\ldots,x_r$ to $K_1$ and denote the resulting field
by $K_2\,$. By \cite[Chapter VI,\S10.3, Theorem~1]{Bo} (see also \cite[Lemma~25]{[K]}), there 
is a place
$\lambda_2$ of $K_2$ whose restriction to $K_1$ is the identity, such that $K_2\lambda_2=K_1$ and
$v_{\lambda_2}x_i=\alpha_i$ for $1\leq
i\leq r$, whence $v_{\lambda_2}K_2= \bigoplus_{1\leq i\leq r} \Z\alpha_i=\Gamma$.

\pars
Since $r\geq 1$, Proposition~\ref{trdegmi} shows that $(K_2,\lambda_2)$ admits an immediate extension of
transcendence degree $s-r$. We pick a transcendence basis $x_{r+1},\ldots,x_s$ of this extension and take
$(K_3,\lambda_3)$ to be the immediate subextension which it generates over $(K_2,\lambda_2)$. It follows that
$\lambda_3|_{K_1}= \lambda_2|_{K_1}=\id_{K_1}\,$. We may choose the elements $x_i$ such that
$v_{\lambda_3} x_i>0$, $r+1\leq i\leq s$. We have the same for $1\leq i\leq r$ since all $\alpha_i$ are positive.

\pars
Now we take $(K_4,\lambda_4)$ to be the henselization of $(K_3,\lambda_3)$. Since it is an immediate extension
of $(K_3,\lambda_3)$, which in turn is an immediate extension of $(K_2,\lambda_2)$, we have that $v_{\lambda_4}
K_4 = v_{\lambda_2}K_2 =\Gamma$ and $K_4\lambda_4=K_2\lambda_2=K_1\,$, as well as $\lambda_4|_{K_1}=\id_{K_1}\,$.

\parm
We wish to show that $F$ can be embedded in $K_4$ over $K$. In fact, we find an embedding $\iota$ of $K_1 .F$
over $K_1$ in $K_4$ as follows. We set $t_i^*:=t'_i+x_i\in K_4\,$, $1\leq i\leq s$;
since $v_{\lambda_4} x_i=v_{\lambda_2} x_i=\alpha_i>0$ we have that $\lambda_4(x_i)=0$ and obtain that
$\lambda_4(t_i^*)=t'_i\,$. Using Hensel's Lemma, we lift the simple root $y'$ of $f(\mbb{t}',Y)$ to an element
$y^*\in K_4$ which satisfies $f(\mbb{t}^*,y^*)=0$ and $\lambda_4(y^*)=y'$.

By construction, $t_1^*,\ldots,t_s^*$ are algebraically independent over $K_1\,$, so we obtain the desired
embedding by setting $\iota(t_i)=t_i^*$ and $\iota(y)=y^*$. Then we take $\lambda$ to be the restriction of
$\lambda_4\circ\iota$ to $F$. As $x_i=t_i^*-t'_i\in \iota(K_1.F)$, the value group of $\lambda_4\circ\iota$ on
$K_1.F$ is equal to $\Gamma$ and consequently, $v_\lambda F\subseteq\Gamma$. As $K_1.F|F$ is finite, so is
$\Gamma/v_\lambda F$.

The restriction of $\lambda$ to $K$ is the identity because the same holds for $\lambda_4$ and $\lambda$ is a
restriction of $\lambda_4\circ\iota$. Hence, $\lambda\in S(F|K)$. Further, $F\lambda\subseteq K_4\lambda_4=K_1\,$,
hence $F\lambda|K$ is finite.

We have that $\lambda(t_j)=\lambda_4(\iota(t_j))=\lambda_4(t_j^*)=t_j'$ for $1\leq j\leq s$, and $\lambda(y)=
\lambda_4(\iota(y))=\lambda_4(y^*)=y'$, whence $\lambda(g_i(\mbb{t},y))=g_i(\mbb{t'},y')$ and $\lambda(h_i
(\mbb{t}))=h_i(\mbb{t'})\ne 0$ for $1\leq i\leq m$. Therefore, $\lambda(a_i)\ne \infty$ for all $i$. By including
also $a_i^{-1}$ in the list for each $i$, we obtain in addition that $\lambda(a_i)\ne 0$ for all $i$.

\pars
Now suppose that we have already constructed places $\lambda_1,\ldots,\lambda_k\in S(F|K)$ which are finite on
$a_1,\ldots,a_m\,$ and satisfy all additional assertions. Since $\trdeg F|K\geq 1$ by
our general assumption, but $F\lambda_j|K$ is algebraic, the places $\lambda_j$ are nontrivial.
Hence there are elements $a_{m+j}\in F$ such that $\lambda_j(a_{m+j})=\infty$ for $1\leq j\leq k$. As shown above,
there exists a place $\lambda$ which is finite on $a_1,\ldots,a_{m+k}$ and satisfies all additional assertions.
It follows that $\lambda(a_{m+j})\ne \infty =\lambda_j(a_{m+j})$ and hence
$\lambda$ is not equivalent to $\lambda_j$ for $1\leq j\leq k$. This shows that there are infinitely
many nonequivalent places which satisfy all assertions of the first part of our theorem.

\parm
If $K$ is existentially closed in $F$, then $F|K$ is separable (see \cite[Lemma~5.3]{[K7]}). In
this case, the proof proceeds as above with $K$ in place of $K_0$ and $F$ in place of $F'$. 
We then have that $K_1=K$,
which implies that $F\lambda=K$. We also have that $t_i'\in K$ for all $i$, which yields that $x_i\in
K(\mbb{t}^*)\subseteq \iota(F)$. As a consequence, $\Gamma\subseteq v_\lambda F$, so that $v_\lambda F=\Gamma$.
\end{proof}

\parb
Now we are ready for the proof of Theorem~\ref{MKPZ}.
\sn
{\it Proof of Theorem~\ref{MKPZ}:}
\n
Assume the setting as in the statement of our theorem.
We write $\Gamma=\bigoplus_{1\leq i\leq r} \Z\alpha_i$ with $\alpha_i>0$.
We break our proof into several parts.

\mn
{\bf Part I: We will first assume that $\wp$ is a nontrivial place.}

\pars
As in the proof of Theorem~\ref{exrp} we choose the fields $K_0$ and $F'$, and the elements 
$t_1,\ldots,t_s,y$, $g_i$ and $h_i$ satisfying (\ref{a_i}) and (\ref{satsent}). We consider 
the place 
$\xi$ extended from $F$ to $F'=K_0.F$. Then we take $\xi_0$ to be the restriction of $\xi$ to
$K_0\,$. Note that $\xi_0$ is an extension of $\wp$ and that $F'\xi$ is algebraic over $F\xi$. 
For every $i$ such that $\xi(a_i)\in K\wp$ we can choose $a'_i\in K$ such that
\[
\xi(a_i) \>=\> \wp(a'_i) \>=\> \xi(a'_i)\>.
\]

As an extension of $\wp$, also $\xi_0$ is nontrivial. Therefore, we can apply part 2) of
Theorem~\ref{realexcl} to obtain that $(K_0,\xi_0)$ is existentially closed in $(F',\xi)$. Hence
there exist elements
\[
t'_1\,,\ldots,\,t'_s\,,\, y' \in K_0
\]
such that for $1\leq i \leq m$,
\[\begin{array}{ll}
   (i) & f(\mbb{t}',y') = 0 \mbox{\ \ and\ \ }
         \frac{\displaystyle\partial f}{\displaystyle\partial Y}
         (\mbb{t}',y')\not= 0\>,\\[0.3cm]
   (ii) & g_i(\mbb{t}',y')\ne 0,\; h_i(\mbb{t}') \not= 0\>,\\[0.3cm]
   (iii) & v_{\xi_0} g_i(\mbb{t}',y') \geq v_{\xi_0} h_i(\mbb{t}') \mbox{ \ \ if \ } a_i\in \cO_\xi\>,\\[0.3cm]
   (iv) & v_{\xi_0} g_i(\mbb{t}',y') > v_{\xi_0} h_i(\mbb{t}') \mbox{ \ \ if \ } a_i\in \cM_\xi\>,\\[0.3cm]
   (v) & v_{\xi_0}\left(\frac{\displaystyle g_i(\mbb{t}',y')} {\displaystyle
         h_i(\mbb{t}')} \,-\, a'_i\right) \> >\> 0 \mbox{ \ \ if \ } \xi(a_i)\in K\wp \>,
\end{array}\]
since these assertions are true in $F'$ for $\mbb{t},y$ in place of $\mbb{t}',y'$ and $v_\xi$ 
in place of $v_{\xi_0}\,$.

\pars
Now let $K_1$ be as in the proof of Theorem~\ref{exrp}, and let $\wp_1$ denote the restriction 
of $\xi_0$ to $K_1\,$. As before, $K_1$ is a finite extension of $K$ and $\wp_1$ is an extension 
of $\wp$. The extension $K_4$ of $K_1$ with the $K_1$-rational place $\lambda_4$ and the 
embedding $\iota$ of $F'$ over $K_1$ in $K_4$ are constructed as in the proof of
Theorem~\ref{exrp}. As before, we obtain $\lambda\in S(F|K)$ with $v_\lambda F\subseteq\Gamma$,
$(\Gamma:v_\lambda F)$ finite and $F\lambda|K$ finite. We take $\wp'$ to be the restriction of 
$\wp_1$ to $F\lambda$. Then assertions {\it (a)} and {\it (b)} of our theorem are satisfied.

We still have to check assertions {\it (c)}, {\it (d)}, {\it (e)} and {\it (f)} on the elements $a_i\,$. Since
$\lambda_4(t_i^*)=t'_i$ and $\lambda_4(y^*)=y'$, we have that
\[
\lambda(g(\mbb{t},y)) \>=\> \lambda_4(\iota(g(\mbb{t},y))) \>=\> \lambda_4(g(\mbb{t^*},y^*)) \>=\> g(t',y')
\]
for every polynomial $g\in K_1[X_1,\ldots,X_s,Y]$. Consequently, using that $h_i(t')\ne 0$ by {\it (ii)},
\[
\lambda(a_i)\>=\>\lambda\left(\frac{g_i(\mbb{t}^*,y^*)}{h_i(\mbb{t}^*)}\right)\>=\>
\frac{g_i(\mbb{t}',y')}{h_i(\mbb{t}')}\>.
\]
Hence {\it (ii)}, {\it (iii)} and {\it (iv)} imply that {\it (c)}, {\it (d)} and {\it (f)} hold (note that
$\lambda(a_i)\in \cO_{\wp'}$ implies that $a_i\in \cO_{\wp'\circ\lambda}=\cO_{\xi'}$ and that $\lambda(a_i)\in
\cM_{\wp'}$ implies that $a_i\in \cM_{\wp'\circ\lambda}=\cM_{\xi'}$). If $\xi(a_i)\in K\wp$, then by {\it (v)},
\begin{eqnarray*}
\xi'(a_i) &=& \wp'(\lambda(a_i))\>=\>\wp'\left(\frac{g_i(\mbb{t}',y')}{h_i(\mbb{t}')}\right) \>=\>
\xi_0\left(\frac{g_i(\mbb{t}',y')}{h_i(\mbb{t}')}\right) \\
&=& \xi_0\left(a'_i\right)\>=\>\xi\left(a'_i\right)\>=\>\xi\left(a_i\right)\>,
\end{eqnarray*}
which shows that also assertion {\it (e)} holds.

\bn
{\bf Part II: We will now assume that $\wp$ is trivial.} In this case we can assume that 
$\wp=\id_K$ since otherwise we replace $\xi$ by $\xi\circ\sigma$ where $\sigma$ is any
monomorphism on $F$ which extends $\wp^{-1}$. We then also choose every extension $\wp'$ of 
$\wp$ to be the identity. Further, we have that $\cO_\wp=K$ and $\cM_\wp=\{0\}$.

\sn
{\bf Part II.1: First we discuss the case where the place $\xi$ is trivial.} Then $\xi$ is a monomorphism and we
may assume
that $\xi|_F=\id_F$ since otherwise, we apply the following proof to $F\xi$ and $\xi(a_i)$ in place of $F$ and
$a_i$ and then replace the places $\lambda$ of $F\xi$ that we obtain by the places $\lambda\circ\xi$ of $F$.

Since $\xi$ is trivial, we have that $\cO_\xi=F$ and $\cM_\xi=\{0\}$. Hence assertion {\it (d)} of our
theorem is satisfied for every $\lambda\in S(F|K)$ because $a_i\in \cM_\xi$ would imply that $a_i=0$, contrary
to our choice of the elements $a_i\,$. Also assertion {\it (e)} is always satisfied, as the condition
$\xi(a_i)\in K\wp$ means that $a_i\in K$, whence $\xi'(a_i)=\wp(a_i)=a_i=\xi(a_i)$ as $\lambda$ and $\wp$ are
trivial on $K$.

For any choice of finitely  many elements $a_1,\ldots,a_m\in F$, Theorem~\ref{exrp}
shows the existence of infinitely many places $\lambda$ which satisfy assertions {\it (a)} and {\it (b)}
as well as $\lambda(a_i)\ne 0,\infty$ for $1\leq i\leq m$. The latter implies that they also satisfy assertions
{\it (c)} and {\it (f)}.

\mn
{\bf Part II.2: Now we deal with the case of $\xi$ being nontrivial.}
\sn
{\bf Part II.2a:} We wish to satisfy assertions {\it (d)} and {\it (e)}, but not necessarily assertion {\it (f)}.

\pars
Assume first that $\trdeg F|K=1$. We claim that $\lambda=\xi$ satisfies assertions {\it (a)}--{\it (e)}. Indeed,
{\it (a)} and {\it (b)} are satisfied since $\xi$ is a nontrivial place of the function field $F|K$ of
transcendence degree 1 which is trivial on $K$. As indicated before, we choose $\wp'$ on $F\xi$ to be the
identity, so we have that $\xi'=\xi$, $\cO_{\wp'}=F\xi$ and $\cO_{\xi'}=\cO_\xi\,$. Hence if $a_i\in\cO_\xi\,$,
then $\lambda(a_i)=\xi(a_i)\in \cO_{\wp'}$ and $a_i\in\cO_{\xi'}\,$, that is, $\lambda$ also satisfies assertion
{\it (c)}. Likewise, we have that $\cM_{\wp'}=\{0\}$ and $\cM_{\xi'}=\cM_\xi\,$. Hence if $a_i\in\cM_\xi\,$, then
$\lambda(a_i)=\xi(a_i)=0\in \cM_{\wp'}$ and $a_i\in\cM_{\xi'}\,$, that is, $\lambda$ also satisfies assertion
{\it (d)}. Further, $\xi'(a_i)=\xi(a_i)$, so also assertion {\it (e)} is satisfied.

\pars
Assume now that $\trdeg F|K>1$. Since $\xi$ is nontrivial, there is some $x\in F$ such that 
$\xi(x)=0$. We denote the $x$-adic place of $K(x)$ by $\xi_x\,$. We apply the already proven part 
of our theorem to the function field $F|K(x)$, with $\wp$ replaced by $\xi_x\,$, to obtain a 
place $\lambda'\in S(F|K(x))$ such that, with $\xi_x$ extended to $F\lambda'$ and 
$\lambda:=\xi_x\circ\lambda'\in S(F|K)$,
\sn
(a$'$)\ \ $F\lambda'$ is a finite extension of $K(x)$,
\sn
(b$'$)\ \ $v_{\lambda'}F\subseteq\Gamma'$ with $(\Gamma':v_{\lambda'} F)$ finite,
\sn
(c$'$) \ if $a_i\in \cO_\xi\,$, then $\lambda'(a_i)\in \cO_{\xi_x}$ and, consequently, 
$a_i\in \cO_{\lambda}\,$,
\sn
(d$'$) \ if $a_i\in \cM_\xi\,$, then $\lambda'(a_i)\in \cM_{\xi_x}$ and, consequently, 
$a_i\in \cM_{\lambda}\,$,
\sn
(e$'$) \ if $\xi(a_i)\in K(x)\xi_x=K$, then $\lambda(a_i) = \xi(a_i)$ for $1\leq i \leq m$,
\sn
(f$\,'$) \ $\lambda'(a_i)\ne 0,\infty$ \ for $1\leq i \leq m$.
\mn
Now (a$'$) implies that $F\lambda=(F\lambda')\xi_x$ is a finite extension of $K(x)\xi_x=K$, so
assertion
{\it (a)} of our theorem is satisfied. Since $\trdeg F\lambda'|K=\trdeg K(x)|K=1$, we obtain that
$v_{\xi_x}(F\lambda')=\Z$, so that $v_{\xi_x} F$ is the lexicographic product of $v_{\lambda'}F$
with $\Z$, which by (b$'$) is a subgroup of $\Gamma$, with $(\Gamma:v_{\lambda} F)$ finite.

To see that assertions {\it (c)} and {\it (d)} of our theorem are satisfied, we recall that we
take the extension $\wp'$ of the trivial place $\wp$ of $K$ to $F\lambda$ to be the identity.
Consequently, $\cO_{\lambda}=\cO_{\wp'\circ\lambda}$ and $\cM_{\lambda}=\cM_{\wp'\circ\lambda}$.
To see that assertion {\it (e)} of our theorem
is satisfied, we use statement (e$'$) above and observe that $\xi'=\wp'\circ\lambda=\lambda$.

\mn
{\bf Part II.2b:} We wish to satisfy assertion {\it (f)}, but not necessarily assertions 
{\it (d)} and {\it (e)}. In the present setting where $\wp$ is trivial, assertion {\it (c)}
follows directly from assertion {\it (f)}. Hence as a matter of fact,
the given place $\xi$ does not play any role. Therefore we obtain infinitely many places 
$\lambda$ with the required properties by just applying Theorem~\ref{exrp}.
\qed

\mn
For the next proof, we will use a version of the Ax-Kochen-Ershov Theorem for 
\bfind{tame fields} as presented in \cite{[K7]}. These are henselian 
valued fields $(K,v)$ whose 
absolute ramification field is algebraically closed. Here, the \bfind{absolute ramification field}
of $(K,v)$ with respect to an extension of the valuation $v$ to the separable algebraic closure
$K\sep$ of $K$ is the ramification field of the extension $(K\sep|K,v)$.

Further, we need the following notation.
If $E$ is any field, we will denote by $E^{1/p^{\infty}}$
its perfect hull (which is equal to $E$ if $\chara E=0$).
\sn
{\it Proof of Proposition~\ref{MKPZa1}:}
\n
{\bf We will first assume that $\wp$ is nontrivial}. We modify the proof of Theorem~\ref{MKPZ}
for this case as follows. We take $(L,\xi)$ to be a maximal algebraic extension of $(F,\xi)$ 
with the property of having
$(F\xi)^{1/p^{\infty}}=(K\wp)^{1/p^{\infty}}$ as its residue field. Then $(L,\xi)$
will have a divisible value group. For the construction of such an extension, see Section~2.3 of
\cite{[K5]}. Further, $(L,\xi)$ is algebraically maximal (i.e., does not admit nontrivial
immediate algebraic extensions) and therefore, it is a tame field by \cite[Theorem~3.2]{[K7]}.

This time we take $K_0$ to be the relative algebraic closure of $K$ in $L$, and $\xi_0$ the restriction of $\xi$
to $K_0\,$; as before, $\xi_0$ is an extension of $\wp$. Since $L\xi=(K\wp)^{1/p^{\infty}}$ is algebraic
over $K\wp$, \cite[Lemma~3.7]{[K7]} shows that $(K_0,\xi_0)$ is a tame field
with $K_0\xi_0=L\xi= (K\wp)^{1/p^{\infty}}$ and $v_{\xi_0}K_0$ equal to the divisible hull of $v_\wp K$.

Since a divisible ordered abelian group is existentially closed in every ordered
abelian group extension by part 5) of Theorem~\ref{realexcl}, and since $\xi_0$ is nontrivial, 
we can apply \cite[Theorem~1.4]{[K7]} to obtain that
$(K_0,\xi_0)$ is existentially closed in $(L,\xi)$. By \cite[Lemma~3.1]{[K7]}, the tame field
$K_0$ is perfect, hence again $K_0.F|K_0$ is separably generated.

From here, the construction proceeds as before. Since $K_1|K$ is finite and $K_1\wp_1$ is
contained in the purely inseparable extension $L\xi$ of $K\wp$, we conclude that $K_1\wp_1$ 
is a finite purely inseparable extension of $K\wp$. Since $\iota(F)\subseteq K_4$, we
have that $F\xi'\subseteq (K_4\lambda_4)\wp_1 = K_1\wp_1\,$. Therefore, $F\xi'|K\wp$ is a
finite purely inseparable extension.

\mn
{\bf Now we assume that $\wp$ is trivial and that $\trdeg F|K>1$}.
By our general assumption on function fields $F|K$, we have that $F\ne K$. Hence if both $\wp$
and $\xi$ are trivial, the condition $F\xi=K\wp$ is never satisfied, and therefore, the assertion
of our proposition is trivially true. Therefore, we now assume that $\xi$ is nontrivial.

\pars
From our assumption that $F\xi=K\wp$ 
it follows that $\xi$ is nontrivial. With the element $x$ chosen as in the corresponding part 
of the proof of Theorem~\ref{MKPZ}, we have that $K(x)\xi_x=K\wp=F\xi$. Then the condition of
our proposition is satisfied for $(K(x),\xi_x)$ in place of $(K,\wp)$. Since $\xi_x$
is nontrivial, from the already proven part of our proposition we infer that in addition to the
results of Theorem~\ref{MKPZ} we can choose $\lambda'$ such that $F(\xi_x\circ\lambda')=
(F\lambda')\xi_x$ is a finite purely inseparable extension of $K(x)\xi_x$. As 
$\xi'=\wp'\circ\xi_x\circ\lambda'= \xi_x\circ\lambda'$ and $K(x)\xi_x=K=K\wp$, this yields 
that $F\xi'|K\wp$ is a finite purely inseparable extension.
\qed

\mn
{\it Proof of Proposition~\ref{MKPZa2}:}
\n
Assume that $(K,\wp)$ is existentially closed in $(F,\xi)$. In the case of nontrivial $\wp$, 
we modify the proof of Theorem~\ref{MKPZ} by setting $K_0=K$. Then we will also have that 
$K_1=K$. The further construction proceeds as in the proof of Theorem~\ref{MKPZ}, yielding a 
place $\lambda$ such that $F\lambda=K$.

In the case of trivial $\wp$, the hypothesis yields that also $\xi$ is trivial, and the
corresponding part of the proof of Theorem~\ref{MKPZ} can be combined with the second
assertion of Theorem~\ref{exrp}. 

In both cases the places $\lambda$ constructed satisfy $F\lambda=K$, which gives us that 
$\wp'=\wp$ and $F\xi'=K\wp$. Further, one shows as in the proof of Theorem~\ref{exrp} that
$v_\lambda F = \Gamma$.
\qed

\mn
{\it Proof of Theorem~\ref{MKPZr}:}
\n
Under the assumptions of Theorem~\ref{MKPZr} we know from part 4) of Theorem~\ref{realexcl} 
that in the language of rings with relation symbols for an ordering and a
valuation, $K$ is existentially closed in $F$. Hence from Proposition~\ref{MKPZa2}, we obtain
places $\lambda$ of $F$ such that in addition to the assertions of Theorem~\ref{MKPZ}, we have
that $F\lambda=R$, $v_{\lambda}F=\Gamma$, $F\xi'=R\wp$ and $\wp'=\wp$.

\pars
To prove assertion {\it (c')}, we assume that $a_i>0$ and $\xi(a_i)\ne 0,\infty$ for some $i$.
This means that $a_i,a_i^{-1}\in\cO_\xi\,$. Hence, in view of assertion {\it (c)} we can choose 
$\lambda$ such that $\lambda(a_i),\, \lambda(a_i^{-1})\in\cO_\wp\,$, so $\xi'(a_i)=
\wp(\lambda(a_i)) \ne 0,\infty$. Since $\lambda(a_i)>0$ and $\wp$ is compatible with the ordering 
on $K$, this implies that $\xi'(a_i)= \wp(\lambda(a_i))>0$. This proves assertion {\it (c')}.

Since $\xi$ is compatible with the ordering on $F$, $\infty\ne\xi(a_i)>0$ implies that $a_i>0$. 
Thus assertion {\it (c')} implies that $\xi'(a_i)>0$.

\parm
It remains to deal with assertion {\it (g)}. We have to show that in all cases where we can get 
$\lambda$ to satisfy assertion {\it (f)}, we can also get it to satisfy assertion {\it (g)}. 
To this end, we replace $F$ by a larger ordered function field in which every positive $a_i$ is 
a square. When $\lambda$ satisfies assertion {\it (f)}, we have that $\lambda(a_i)\ne 0,\infty$.
If $a_i$ is a square, it then follows that also $\lambda(a_i)$ is a nonzero
square, hence positive.
\qed

\mn
{\it Proof of Proposition~\ref{MKPZa3}:}
\n
%
It suffices to show the assertion for the places $\lambda$. This is seen as follows. 
The valuation ring of $\lambda$ is an overring of $\wp'\circ\lambda$ and the
overrings of a valuation ring in a field are linearly ordered by inclusion. 
Hence if $\lambda_1$ and $\lambda_2$ are such that $\wp'\circ\lambda_1$ and $\wp'\circ\lambda_2$
have the same valuation ring, then the valuation rings of $\lambda_1$ and $\lambda_2$ must be
comparable by inclusion. But since $F\lambda_1$ and $F\lambda_2$ are
algebraic over $K$, this implies that the two valuation rings are equal.

\pars
In all cases where the constructed places $\lambda$ satisfy assertion {\it (f)} of our theorem,
the argument given in the proof of Theorem~\ref{exrp} shows the existence of infinitely many
nonequivalent places $\lambda$.

\pars
It remains to prove our assertion in the case where $\wp$ is trivial while $\xi$ is not and 
$\trdeg F|K>1$, and we want the places to satisfy assertions {\it (d)} and {\it (e)}. 
In this case, in the corresponding part of the proof of Theorem~\ref{MKPZ}, we constructed places
$\lambda'$ satisfying assertion (f$\,'$), hence by what we just said, there are infinitely many
nonequivalent such places $\lambda'$. By our above argument (with $\lambda$ replaced by 
$\lambda'$ and $\wp$ replaced by $\xi_x$), also the
resulting places $\lambda=\xi_x\circ\lambda'$ are nonequivalent.
\qed

\mn
%
%
\section{Applications to topologies and holomorphy rings}

In this section, $F$  always denotes a formally real function field over a real closed base field $R$. 

\subsection{Sets of real places and their topologies}
\mbox{ }\sn
From Theorem~\ref{MKPZr} we will deduce:
\begin{proposition}                        \label{dense}    
1) The set $M_R(F)$ is dense in $M(F)$ with respect to $\TopM(F)$.
\sn
2) If in addition $R$ is non-archimedean, then every nonempty intersection of an open set in 
the Zariski patch topology of $M(F)$ with an open set in $\TopM(F)$ contains infinitely many
places from $M_R(F)$.
\end{proposition}
\begin{proof}
Assume that there is an $\R$-place $\xi \in U(f_1,...,f_m)$ and choose
a compatible ordering $<_{\xi}$ on $F$. Then there are positive rational numbers $q_1$ and $q_2$ such that
\[
q_1 <_{\xi} f_i <_{\xi} q_2 \;  \mbox{ for } 1\leq i\leq m\>.
\]
Using Theorem~\ref{MKPZr}, we obtain a place $\lambda\in M(F|R)$ such that
\[
q_1 <\lambda(f_i) < q_2
\]
in $R$. Composing $\lambda$ with $\xi_R$ we obtain
\[
q_1 \leq \xi_R\circ \lambda(f_i) \leq q_2\>,
\]
which shows that the $\R$-place $\xi_R\circ \lambda$ is in $U(f_1,...,f_m)$. This proves the 
first assertion of Proposition~\ref{dense}.

\pars
In order to prove the second assertion, assume in addition that $R$ is non-archi\-medean. Then 
$\xi_R$ is a
nontrivial place. Further, take elements $a_1,\ldots,a_{k+\ell+m}\in F$ and an $\R$-place $\xi$
of $F$ such that $a_1,\ldots,a_k\in {\cO}_\xi\,$, $a_{k+1},\ldots,a_{k+\ell}\in \cM_\xi$ and
$\xi(a_{k+\ell+1})>0,\ldots,\xi(a_{k+\ell+m})>0$. Note that $\xi|_R=\xi_R$. Hence by
Theorem~\ref{MKPZr} in connection with Proposition~\ref{MKPZa3},
there are infinitely many $R$-rational places $\lambda$ of $F$ and places 
$\xi'=\xi_R\circ\lambda$ such that:
\sn
(1) \ $a_1,\ldots,a_k\in {\cO}_{\xi'}$ and $a_{k+1},\ldots,a_{k+\ell}\in {\cM}_{\xi'}\,$;
\sn
(2) \ $\xi'(a_{k+\ell+1})>0,\ldots,\xi'(a_{k+\ell+m})>0\,$.
\end{proof}

\parm
We observe the following equivalences that hold for all $a\in F$:
\begin{eqnarray}
\lambda(a)>0 & \Leftrightarrow & \lambda\left(\frac{a}{1+a^2}\right)>0\>,  \label{eqa}\\
\lambda (a)\ne 0,\infty & \Leftrightarrow & \lambda\left(\frac{a^2}{1+a^2}\right)\> > \> 0\>.   \label{eqa^2}
\end{eqnarray}

\parm
We introduce a topology $\TopM(F|R)$ on $M(F|R)$ through the basic open sets
\[
V(f_1,...,f_m)\>:=\> \{\lambda \in M(F|R)\mid \lambda(f_i)>0 \mbox{ for } 1\leq i\leq m\}
\]
where $f_1,...,f_m \in H(F|R)$. Note that $\frac{f_i}{1+f_i^2} \in H(F)\subseteq H(F|R)$. Using the equivalence
(\ref{eqa}) we can thus replace the condition ``$f_1,...,f_m \in H(F|R)$'' by ``$f_1,...,f_m \in H(F)$'' without
changing the collection of basic sets.
\begin{proposition}                    \label{TopMF|R}
Consider elements $f_1,\ldots,f_k\in H(F)$ and nonzero elements $a_1,\ldots,a_\ell\in F$. If the basic open set $V(f_1,\ldots,f_k)$ of $\TopM(F|R)$ is
non\-empty, then there are infinitely many places in
\begin{equation}                  \label{HarZarset}
\{\lambda\in V(f_1,\ldots,f_k)\mid \lambda(a_j)\ne 0,\infty \mbox{ for } 1\leq j\leq \ell\}\>.
\end{equation}
\end{proposition}
\begin{proof}
Take $\lambda\in V(f_1,\ldots,f_k)$. Then
$\lambda$ is an $R$-rational place with $\lambda(f_i)>0$ for $1\leq i \leq k$ and therefore $F$ admits an ordering $P$ which is compatible with $\lambda$
and under which $f_1,\ldots,f_k$ are positive. We next pass to the function field $F'=F(a_{l+1},\dots,a_{l+k})$ where $f_i=a_{l+i}^2$ for $i=1 \leq i \leq k$. The ordering $P$ of $F$ can be extended to $F'$, so $F'$ is also a formally real function field over $R$. The extension $F'|F$ is finite and we note that all $a_i, i=1,\dots,l+k$, are nonzero elements.

Since $R$ is real closed and $F'$ is formally real, we know that the field $R$ is 
existentially closed in $F'$. By Theorem~\ref{exrp} there are infinitely many places in 
$M(F'|R)$ which do not take the value $\infty$ on $a_1,\ldots,a_{\ell+k},a_1^{-1},
\ldots,a_{\ell+k}^{-1}$. Hence they do not take the values $0,\infty$ on
$a_1,\ldots,a_{\ell+k}$. In particular, they send $f_1,\ldots,f_k$ to squares 
$\ne 0,\infty$ which consequently
are positive elements of $R$. The restrictions of these places are places in $M(F|R)$ as 
their residue fields are subfields of finite degree below $R$, hence equal to $R$. So 
they yield the desired places in the set (\ref{HarZarset}).
Indeed, since $F'|F$ is finite, the restriction of infinitely many places of $F'$ yields
infinitely many places of $F$.
\end{proof}

\begin{remark}

Alternatively, one may prove the last proposition by passing to regular affine $R$-algebras 
$A$ with quotient field $F'$ which contain at least all elements $a_1,\dots,a_{\ell+k},
a_1^{-1},\ldots,a_{l+k}^{-1}$. One then applies the Artin-Lang Homomorphism Theorem (see
\cite[4.1.2]{BCR}) and uses the fact that every regular $R$-point is the
center of a rational $R$-place. Yet, as shown in \cite{BCR} the Artin-Lang Theorem 
is a kind of geometric version of model theoretic facts.
\end{remark}

A point $x$ in a topological space is called \bfind{isolated} if the singleton $\{x\}$ is an open set.
\begin{theorem}                                    \label{topspaces}
Let $\xi_R$ denote the natural $\R$-place of $R$.
\sn
1) The mapping $\iota_{F|R}: M(F|R)\to M_R(F)$ defined in (\ref{MFRtoMRF}) is a bijection.
\sn
2)  $\iota_{F|R}$ is a topological embedding of $M(F|R)$ into $M(F)$.
\sn
3) All nonempty open sets in $M(F|R)$, $M(F)$ and $M_R(F)$ are infinite.
\sn
4) In particular, none of the spaces $M(F|R)$, $M(F)$ and $M_R(F)$ admit any isolated points.
\sn
\end{theorem}
\begin{proof}
1): This is a special instance of Proposition \ref{bijective} in the next section.

\mn
2): We first prove that $\iota_{F|R}:M(F|R)\rightarrow M(F)$ is continuous. Take $\lambda\in M(F|R)$ and $f \in
H(F)$ such that $\xi:=\iota_{F|R}(\lambda)=\xi_R\circ\lambda\in U(f)$. Then
there are positive rationals $c, d$ such that $c<\xi(f)<d$. Then also $c<\lambda(f)<d$, so 
\[
\lambda \in V(f-c)\cap V(d-f)\>=:\>V
\]
and $V$ is an open neighbourhood of $\lambda$. We will show that $\iota_{F|R}(V)\subseteq U(f)$.
If $\lambda' \in V$, then $c<\lambda'(f)<d$, whence $c\leq\xi_R \circ\lambda'(f)\leq d$. Thus $\xi_R
\circ\lambda'\in U(f)$.
Hence, $\iota_{F|R}$ is shown to be continuous.

\pars
Next, we prove that $\iota_{F|R}:M(F|R)\rightarrow M_R(F)$ is an open map. To this end, take an arbitrary
subbasic open set $V(f) = \{\lambda\in M(F|R)\mid \lambda(f)>0\}$ where we may take $f \in H(F)$. We have
to show that $\iota_{F|R}(V(f))$ is open in the subspace topology on $M_R(F)$. Take any $\lambda \in V(f)$ and
set $\xi=\xi_R\circ \lambda$.
Then $a:=\lambda(f)\in H(R), a>0$. Set $g:=\frac{af}{a^2+f^2}$. One sees that $g\in H(F)$.
We obtain that $\lambda(g)=\frac 1 2=\xi(g)$ and therefore $\xi\in U(g)\cap M_R(F)$. We want to show that
the whole neighbourhood $U(g)\cap M_R(F)$ of $\xi$ is contained in $\iota_{F|R}(V(f))$.

If $\xi'\in U(g)\cap M_R(F)$, then $\xi'=\xi_R\circ\lambda'$ with $\lambda'\in M(F|R)$, and
$\xi_R(\lambda'(g))=\xi'(g)>0$ implies
that $\lambda'(g)>0$, whence $\lambda'(f)>0$. This yields that $\lambda'\in V(f)$, and the inclusion
$U(g)\cap M_R(F) \subseteq \iota_{F|R}(V(f))$ is proven.

\mn
3): The assertion about $M(F|R)$ follows from Proposition~\ref{TopMF|R}. From this the assertion about $M_R(F)$
follows by part 2) of our theorem, which together with the density of $M_R(F)$ in $M(F)$ (see
Proposition~\ref{dense}) implies the assertion for $M(F)$.

\mn
4): The assertions follow directly from part 3).
\end{proof}

\begin{remark}
Here is an even simpler proof of the fact that $M(F)$ has no isolated points (from which the same follows for
$M_R(F)$ and $M(F|R)$ via the density of $M_R(F)$ in $M(F)$ and part 2) of the above theorem). We have that $F$
is a finite extension of some rational function field $R(x_1,...,x_n)$. Assume that
$\xi$ is an isolated point in $F$, i.e., $U:= \{\xi\}$ is an open subset of $M(F)$. Take the inverse image $V$
of $U$ in the space of orderings of $F$. It is open since $M(F)$ is a quotient space of the space of orderings
of $F$, and it has only finitely many elements by the Baer-Krull Theorem (note that the value group of the
restriction of any place $\xi\in M(F)$ to the real closed field $R$ has divisible value group and as $F|R$ has
finite transcendence degree, it follows that $v_\xi(F)/2v_\xi(F)$ is finite). Consider the set of all orderings on
$R(x_1,...,x_n)$ induced by the orderings in $V$. By the openness of the restriction function for 
orderings (see \cite[Theorem 4.4]{ELW}), this set is open in the space
of orderings of $R(x_1,...,x_n)$. As it contains a finite number of elements, this is impossible, as
\cite[Theorem 10]{C} shows that the space of orderings of $R(x_1,...,x_n)$ does not have isolated points.
\end{remark}

\subsection{Holomorphy rings}                                    \label{sectHr}
\mbox{ }\sn
In the introduction we alluded to the rings 
\[
H(F)B[x_1,\dots,x_n]\>,
\] 
$B$ any real valuation ring of $R$. We will
show that these rings  admit a description as an intersection of valuation rings of a family $\cF$ of composite
places, or in other words: they are the holomorphy ring of this family. This section begins with a general study
of rings which are intersections of families of valuation rings of composite places. It turns out that this
property is closely related to a certain type of Nullstellensatz, a fact which was first observed by H.-W.\
Sch\"ulting, see~\cite[Section~2]{Sch1} in the context considered there.

Two further issues will be discussed in this section.  We look at the existence of minimal representations as an
intersection of valuation rings of composite places. Secondly, a new description
of the real holomorphy ring  $H(F)$ will be presented which can be seen as an analogue to a certain refinement of
Artin's solution of Hilbert's 17th problem.

\parm
Given any subring $D$ of $F$, the relative real holomorphy ring $H(F|D)$ is defined as follows:
\[
H(F|D)\>:=\>\bigcap\{\cO \mid \cO \text{ a real valuation ring of } F \text{ and }
D\subseteq\cO\}\>.
\]

As said above, relative real holomorphy rings are overrings of $H(F)$, hence they are Pr\"ufer rings with $F$ as their field of fractions. This applies to our special case. In addition, we will be using that  Pr\"ufer rings are the intersection of their valuation overrings and that the valuation overrings of $H(F)$ are exactly the valuation rings of the real places of $F$.
We find that
\begin{eqnarray*}
H(F|D)&=&H(F)D\>,\\
H\left(F|B[x_1,\dots,x_n]\right)&=&H(F|B)[x_1,\dots,x_n]\>=\> H(F)B[x_1,\dots,x_n]\>,
\end{eqnarray*}
for any family $x_1,\dots,x_n$ of elements of $F$.

In fact, one checks that the
rings to be compared admit the same set of valuation overrings.

\parm
Note: if a subring $A\subseteq F$ is the intersection of real valuation rings of $F$, then it must contain $H(F)$.
Hence, for a general discussion we will impose the condition
\[
H(F)\>\subseteq\> A
\]
throughout, if not stated otherwise. Under this condition, the ring $A$ is a Pr\"ufer ring. Hence it is the
intersection of all valuation overrings which are real valuation rings as they contain $H(F)$. However, we are
not interested in this sort of presentation of $A$ as an intersection of general real valuation rings. As mentioned before, we
want to study rings $A$ which admit an intersection presentation by valuation rings of composite places.

\pars
The real valuation rings of the base field $R$ are just the overrings of $H(R)=\cO_R$. They 
will be denoted by, say, $B$ and $C$, and their canonical places by $\pi_B$ and $\pi_C\,$.
Recall that if $B$ is a valuation ring of $R$ with maximal ideal $\cM_B\,$, then $\pi_B$ sends
every $a\in \cO_B$ to $a+\cM_B\in \cO_B/\cM_B\,$, and every $a\in R\setminus\cO_B$
to $\infty$. If $\wp$ is a real place on $R$, then $\pi_B$ is equivalent to $\wp$. In particular,
if $B=\cO_R$, then $\pi_B$ is equivalent to $\xi_R$, and if $B=R$, then 
$\pi_B$ is equivalent to ${\rm id}_R$.

These real places are pairwise non-equivalent and they represent the equivalence classes of real 
places on $R$. The objects of central interest in this paper are the \textbf{composite places} 
$\wp\circ \lambda$, $\wp$ any real place on $R$ and $\lambda\in M(F|R)$. They are equivalent to
the places $\pi_B\circ \lambda$, $B$ any real valuation ring of $R$, $\lambda \in M(F|R)$. This
follows from the fact that the valuation ring of $\wp\circ \lambda$ equals $\lambda^{-1}(B)$ 
where $B=\cO_\wp=\cO_{\pi_B}\,$. In addition, $\cO_{\pi_B\circ \lambda}$ has the maximal ideal 
$\lambda^{-1} (\cM_B)$ and its residue field equals the residue field of $B$.

From the following result we then deduce the fact that the family $\pi_B\circ\lambda, B$ a real
valuation ring of $R, \lambda \in M(F|R)$, is a complete system of representatives of the 
family of composite places as defined above. 
\begin{proposition}                        \label{bijective}
Let $B,C$ be real valuation rings of $R$ and $\lambda,\mu \in M(F|R)$. If the places $\pi_B\circ \lambda$ and $\pi_C\circ \mu$ are equivalent, then
\[
B=C, \>\lambda=\mu.
\]
\end{proposition}
\begin{proof}Let $V$ denote the valuation ring of the composite place $\pi_B\circ \lambda$. Then $V\cap R=B$. Hence, $B=C$ follows. Clearly, $V\subseteq \cO_\lambda, V\subseteq \cO_\mu$. Therefore, these two valuation rings are comparable, say $\cO_\lambda\subseteq \cO_\mu$. Pick any $a\in \cO_\mu$. then $\mu(a-\mu(a))=0$. As the maximal
ideal of the larger valuation ring $\cO_\mu$ is contained in the maximal ideal of $\cO_\lambda$, we find that
$\lambda(a-\mu(a))=0$, whence $\lambda=\mu$.
\end{proof}

\vspace{.2cm}
The concepts and results we are presenting in this paper only depend on the equivalence class of the composite places, as one may check. Therefore, it is sufficient to work with the distinguished family 
\[
\cC(F)\>:=\> \{\pi_B\circ\lambda\mid B \text{ real valuation ring of }R, \ \lambda \in M(F|R)\}
\]
of composite places, which we will abbreviate as $\cC$, and for a given ring subring $A$ of $F$ with the set
\[
\cC_A\>:=\> \{\xi\in \cC\mid \xi \text{ finite on } A\}=\{\pi_C\circ \lambda\mid \lambda \in M(F|R),\ \lambda(A)\subseteq C\}\>.
\]
In particular,
\[
\cC_{H(F|B)}\>=\> \{\pi_C\circ\lambda\mid \lambda \in M(F|R), \ B\subseteq C\}\>,
\]
since for each $\lambda\in M(F|R)$ we have that $\lambda(H(F))=H(R)=\cO_R$ and therefore $\lambda(H(F|B))=
\lambda(H(F)B)=B$. Since
\[
\cC_{A[x]}\>=\>\{\xi \in \cC_A\mid \xi(x)\neq \infty \}
\]
and $\lambda(H(F|B)[x])=B[\lambda(x)]$ if $\lambda$ is finite on $H(F|B)[x]$, we obtain that
\begin{equation}                    \label{cCH(F|B)}
\cC_{H(F|B)[x]}\>=\>\{\pi_C\circ\lambda\mid \lambda \in M(F|R), \ B\subseteq C \text{ and } \lambda(x)\in C\}\>.
\end{equation}

We will make use of the following 
\begin{lemma}                                    \label{f.g.}
Assume that $H(F)\subseteq A, \ x_1,\dots,x_n\in F$, and $\mathfrak{a}$ is a finitely generated ideal of $A$. Then:
\sn
1) there is $x\in F$ with $A[x_1,\dots,x_n]=A[x]=A[1+x^2]$,
\sn
2) $\sqrt{\mathfrak{a}}$ is the radical of a principal ideal.
\end{lemma}
\begin{proof}
1): \ We show that $A[x_1,\dots,x_n]=A[1+\sum_1^n x_i^2]$. The inclusion ``$\supseteq$'' is clear, while the
inclusion ``$\subseteq$'' follows from the fact that $\frac{x_i}{1+\sum_k x_k^2}\in H(F)$ for all $i$, and
$H(F)\subseteq A$. We set $x=1+\sum_1^n x_i^2$ and observe that a similar argument shows that $A[x]=A[1+x^2]$
because $\frac{x}{1+x^2}\in H(F)$.
\sn
2): \ Let $\mathfrak{a}=(f_1,\dots,f_n)$. Then $\mathfrak{a}^2=(\sum_1^nf_i^2)$ as $\frac{f_if_j}{\sum_kf_k^2}\in
H(F)\subseteq A$. Now our assertion follows since $\sqrt{\mathfrak{a}}=\sqrt{\mathfrak{a}^2}$.
\end{proof}
\n
In view of part 1) of this lemma, whenever we will consider a finitely generated ring extension $A'$ of $A$ inside $F$, we may
always assume it to be of the form $A'=A[1+x^2]$ for some $x\in F$.

\pars
We say that the ring $A$ satisfies the \bfind{intersection property} if
\[
A\>=\>\bigcap_{\xi\in \cC_A}\cO_\xi\>,
\]
or in other words, if $A$ is the \bfind{holomorphy ring of the family $\cC_A\,$}.

Next, we consider an ideal $\mathfrak{a}$ of $A$, from which we obtain the zero set
\[
V(\mathfrak{a})\>:=\> \{\xi \in \cC_A\mid \xi=0 \text{ on }\mathfrak{a}\}\>.
\]
Likewise, from a subset $V\subseteq \cC_A$ we obtain the vanishing ideal
\[
I(V)\>:=\>\{a\in A\mid \xi(a)=0  \mbox{ for all } \xi\in V\}\>.
\]
Clearly, $\sqrt{\mathfrak{a}}\subseteq I\left(V(\mathfrak{a})\right)$. Following the usual terminology, we say
that the ideal $\mathfrak{a}$ satisfies the \bfind{Nullstellensatz} if
$I\left(V(\mathfrak{a})\right)=\sqrt{\mathfrak{a}}$.

\pars
The following proposition extends Sch\"ulting's result \cite[2.6]{Sch1}.
\begin{proposition}                                  \label{nullstellensatz}
Assume that $H(F)\subseteq A$. Then the following statements are equivalent:
\sn
1) $A$ satisfies the Nullstellensatz for finitely generated ideals,
\sn
2) every finite ring extension $A[x_1,\dots,x_n]$, where $n\in \N, \ x_1,\dots,x_n\in F$, satisfies the Nullstellensatz for
finitely generated ideals,
\sn
3) every finite ring extension $A[x_1,\dots,x_n]$, where $n\in \N, \ x_1,\dots,x_n\in F$ has the intersection property.
\sn
\end{proposition}
\begin{proof}
The implication 2) $\Rightarrow$ 1) is trivial, as $A$ is its own finite ring extension.
Once the equivalence of 1) and 3) is proven for \textbf{all} overrings of $H(F)$, the implication 1) $\Rightarrow$
2) also follows: if 1) holds, then by 3), every finite ring extension of $A' :=A[x_1,\dots,x_n]$ of $A$, being also a finite ring
extension of $A$, has the intersection property, which implies that 1) holds for $A'$.

1) $\Rightarrow$ 3):
As stated after the previous lemma, we know that  
\[
A[x_1,\dots,x_n]\>=\>A[1+x^2]
\]
for some $x\in F$. Set $A'=A[1+x^2]$ and consider any $f\in \bigcap_{\xi\in \cC_{A'}}\cO_\xi$. 
\\ From the definition we deduce that
$$ \xi \in \cC_{A'} \Leftrightarrow \xi \in \cC_A, \ \xi(1+x^2)\neq\infty .$$
Hence, for all $\xi\in \cC_A$, the implication $\xi(1+x^2)\neq\infty\Rightarrow
\xi(1+f^2)\neq\infty$, holds, and so its contraposition
\begin{equation}             \label{contrap}
\xi\left(\frac{1}{1+f^2}\right)=0\>\Rightarrow\> \xi\left(\frac{1}{1+x^2}\right)=0\>.
\end{equation}
We observe that $\frac{1}{1+x^2},\frac{1}{1+f^2}\in H(F)\subseteq A$. Consider the principal $A$-ideal $\mathfrak{a}=
\left(\frac 1 {1+f^2}\right)$. From (\ref{contrap}) we obtain that $\frac{1}{1+x^2}\in
I\left(V(\mathfrak{a})\right)$. We infer from $1)$ that
$\left(\frac{1}{1+x^2}\right)^k=a\frac{1}{1+f^2}$ for some $k\in\N$, $a\in A$ and $1+f^2\in A'=A[1+x^2]$. The Pr\"ufer
ring $A'$ is integrally closed, so $f\in A'$ as was to be shown.
\sn
3) $\Rightarrow$ 1): \ Take any finitely generated ideal $\mathfrak{a}$ of $A$. In the proof of statement 2) of the previous lemma it was shown that $\mathfrak{a}^2=(f)$ for some $f\in A$. In view of $V(\mathfrak{a})=V((f))$ and $\sqrt{\mathfrak{a}}=\sqrt{(f)}$ we may assume that $\mathfrak{a}$ is a
principal ideal $(f)$. Consider $g\in I(V(f))$, so $\xi(f)=0\Rightarrow \xi(g)=0$ holds for all $\xi \in \cC_A\,$.
The composite places which are finite on the extension $A[\frac{1}{g}]$ are just the composite ones which are
finite on $A$ and satisfy $\xi(\frac{1}{g})\neq\infty$. By the contrapositive of the above implication, these
places also satisfy $\xi(\frac{1}{f})\neq\infty$. By assumption, the extension $A[\frac{1}{g}]$ has the intersection
property, so we obtain that $\frac{1}{f}\in A[\frac{1}{g}]$. From this, $g^k\in (f)$ follows for some $k\in \N$.
\end{proof}

\begin{remark}
1) \ That a ring $A$ admits the intersection property does not imply that the Nullstellensatz holds for finitely
generated ideals of $A$. To obtain this implication one really needs the hypothesis for all finitely generated
extensions as above. For an example, take $A=\cO_\lambda$ for some $\lambda \in M(F|R)$. Then $\cC_A=\{\lambda\}$
and the intersection property trivially holds. But if the rank of $\lambda$ is greater than $1$, then there is
$f\in \cO_\lambda$ with $\sqrt{(f)}\neq \cM_\lambda$, while $I(V(f))=\cM_\lambda\,$.

Pick any $x\in F\setminus \cO_\lambda\,$. Then there is no composite place which is finite on the extension
$\cO_\lambda[x]$. Hence this ring does not have the intersection property.

\pars
2) \ Assume that $A$ has the intersection property. Then $\cO_R$ is contained in $A$ since it is contained in
$\cO_\xi$ for all $\xi\in \cC_A\,$. It follows that for every $\lambda\in M(F|R)$ that is finite on $A$,
$\lambda(A)$ is a ring containing $\cO_R$, hence a real valuation ring of $R$.
This leads to the representation
\[
A\>=\>\bigcap\{\cO_{\pi_{\lambda(A)}\circ \lambda}\mid \lambda\in M(F|R)\text{ finite on }A\}\>,
\]
since the valuation rings on the right hand side are the minimal ones among all valuation rings $\cO_\xi$ with
$\xi\in \cC_A\,$.
\end{remark}

\begin{theorem}                                    \label{H(F|B)}
For every real valuation ring $B\subseteq R$, each finite ring extension of $H(F|B)$ within $F$
satisfies the Nullstellensatz for finitely generated ideals and has the intersection property.
\end{theorem}
\begin{proof}
By Proposition~\ref{nullstellensatz}, it suffices to prove that any finite extension $A$ of 
$H(F|B)$ has the intersection property. Due to Lemma \ref{f.g.} we may write $A=H(F|B)[x]$ for
some nonzero $x\in F$. Suppose that there exists $f\in\bigcap_{\xi\in \cC_A}\cO_\xi$ with 
$f\notin A$. 
As said above, $A$ is the intersection of all real valuation rings in which it is contained. 
Hence we find a real place $\xi_0$ with $A\subseteq\cO_{\xi_0}$, $x\in \cO_{\xi_0}$ and 
$f\notin \cO_{\xi_0}\,$. Applying Theorem~\ref{MKPZr} with $\xi_0$ in place of $\xi$ and 
$\wp$ the restriction of $\xi_0$ to $R$, we obtain a place $\lambda\in M(F|R)$ which satisfies
assertions c) and d) of the theorem with $a_1=x$, $a_2=\frac{1}{f}$.
Then $\lambda(x)\in \cO_\wp$ and $\lambda(\frac{1}{f})\in \cM_\wp\,$, 
whence $\lambda(f)\notin \cO_\wp\,$.
We set $C:=\cO_\wp=\cO_{\xi_0}\cap R$, so $\lambda(x)\in C$. Since 
$H(F|B)\subseteq A\subseteq \cO_{\xi_0}\,$,
we also have that $B\subseteq C$. Hence by (\ref{cCH(F|B)}), $\pi_C\circ\lambda\in \cC_A\,$. But
$\lambda(f)\notin \cO_\wp$ implies that $f\notin \cO_{\pi_C\circ \lambda}$, a contradiction 
to our choice of $f$.
\end{proof}

The three distinguished cases of $A=H(F|B)$, $A=H(F)=H(F|\cO_R)$ and $A=H(F|R)$ deserve special
attention:
\begin{eqnarray}
H(F|B)&=&\bigcap \{\cO_{\pi_B\circ \lambda}\mid \lambda\in M(F|R)\}\>,\\
H(F)&=&\bigcap\{\cO_{\xi_R\circ \lambda}\mid\lambda\in M(F|R)\}     \label{H(F)}\\
&=& \bigcap\{\cO_{\xi}\mid\xi\in M_R(F)\}\>,\nonumber \\
H(F|R)&=&\bigcap\{\cO_\lambda\mid\lambda\in M(F|R)\}\>.
\end{eqnarray}
The presentation of $H(F|B)$ immediately yields the equality $H(F|B)\cap R=B$; in other words:
$H(F|B)$ is the smallest relative real holomorphy which extends the real valuation ring $B$. It
should be mentioned that this fact can be deduced more easily. In fact, choose any ordering $P$ 
of $F$ and consider the convex closure $V$ of $\Q$ in $F$ with respect to $P$. Then $V$ is a real
valuation ring of $F$ extending $\cO_R$, and one deduces that the real valuation ring $VB$ 
extends $B$. This implies the nontrivial inclusion $H(F|R)\cap R\subseteq B$.

As listed above, \textbf{$H(F)$ is the intersection of the family of valuation rings of the real places in
$M_R(F)$.} This is a straightforward and appealing geometric generalization of the situation in case of $R=\R$.

\parm
In what follows we address the question whether there are minimal representations for the relative real holomorphy
rings $H(F|B)$ of the type above. More precisely, we will study subfamilies $\cF\subseteq M(F|R)$
such that $H(F|B)=\bigcap_{\lambda \in \cF}\cO_{\pi_B\circ \lambda}$ and look at the existence of minimal
families $\cF$.
This is a topic dealt with by Sch\"ulting in \cite[3.13]{BS} and \cite[1.3 ff.]{Sch2}
for the case $B=R$. His results are incorporated. More generally, we
allow $B$ to range over all real valuation rings of the base field $R$.
\begin{theorem}                                  \label{Becker}
Let $B$, $C$ be two real valuation rings of $R$.  Then we have:
\sn
1) If $H(F|B)=H(F|C)$, then $B=C$;
\sn
2) if $B\neq R$, then the following statements are equivalent for each subset $\cF$ of $M(F|R)$:
\pars
(a) $H(F|B)\>=\> \bigcap_{\lambda \in \cF} \cO_{\pi_B\circ \lambda}\>$,
\pars
(b) $\cF$ is dense in $M(F|R)$;
\sn
3) if $B\neq R$, then there is no representation of the form (a) with a minimal $\cF$;
\sn
4) $H(F|C)$ admits a representation of the form (a) with a minimal $\cF$ if and only if $C=R$ and $\trdeg F|R=1$.
In the case of a minimal representation we necessarily have that $\cF=M(F|R)$.
\end{theorem}
\begin{proof}
1):
We showed above that $H(F|B)\cap R=R$; the same holds for $C$.
\sn
2): Assume that $\cF$ is not dense in $M(F|R)$. Hence by part 2) of Theorem~\ref{topspaces},
$\iota_{F|R}(\cF)$ is not dense in $M_R(F)$ and thus also not in $M(F)$. Let $N$ be the closure of
$\iota_{F|R}(\cF)$ in $M(F)$ and take $\eta \in M(F)\setminus N$. By the Separation Criterion given in
\cite[Proposition~9.13]{Lam}, there is $f \in H(F)$ such that $N\subseteq U(-f)$ and $\eta \in U(f)$.
Since $M_R(F)=\iota_{F|R}(M(F|R))$ is dense in $M(F)$ by Proposition~\ref{dense}, there is $\lambda_0 \in M(F|R)$
such that $\xi_R \circ \lambda_0 \in U(f)$ and thus $a:=\lambda_0(f)$ is an element of the set $E^+(R)$
of positive units of $\cO_R\,$. For $\lambda \in \cF$ we have that $-\lambda (f) \in E^+(R)$.
Define $g := \frac {1}{a-f}$. We have that $\lambda_0(g) = \infty$ and therefore $g \notin
\cO_{\pi_B\circ \lambda_0}$, whence $g\notin H(F|B)$. But for $\lambda \in \cF$ we have $\lambda(g) \in E^+(R)
\subseteq \cO_R \subseteq B$ and therefore $g \in \cO_{\pi_B\circ \lambda}$. Hence (a) cannot be true.
This proves that (a) implies (b).

Now we prove that (b) implies (a). We have that $H(F|B)\>\subseteq\> \bigcap_{\lambda \in \cF} \cO_{\pi_B\circ
\lambda}\>$. Suppose that equality does not hold. Then there is some $f \in F$ and $\lambda_0 \in M(F|R)$ such
that $\lambda_0(f)\notin B$ but $\lambda(f)\in B$ for all $\lambda\in\cF$. Then either $\lambda_0(f)=\infty$ or
$\lambda_0(f)\in R\setminus B$. In the first case we choose $a\in R\setminus B$ with $a>0$ (note that $B\subsetneq
R$ since $B\subsetneq C$). In the second case
we choose $a$ such that $2a=\lambda_0(f)$; switching $f$ to $-f$ if necessary, we may again assume that $a>0$.
In both cases we see that $-a<b<a$ for all $b\in B$ as $B$ is convex under the ordering $<$ of $R$.
We consider the set
\[
S\>:=\>\{\lambda\in M(F|R)\mid \lambda(f)=\infty\mbox{ or } |\lambda(f)|>a\}
\]
and note that it contains $\lambda_0\,$. With $g:=\frac{f}{1+f^2}\in H(F)$, the condition defining $S$ holds if
and only if $\lambda(g)=0$ or $|\lambda(g)|<\frac{1}{1+a^2}=:b$, which means that $-b<\lambda(g)<b$. This shows
that $S$ is an open subset of $M(F|R)$. By the density of $\cF$ in $M(F|R)$ there is some $\lambda\in \cF\cap S$.
Consequently, $\lambda(f)=\infty$ or $|\lambda(f)|>a$ so that $\lambda(f)\notin B$, a contradiction to our
assumption on $\lambda(f)$. Thus equality, and hence (a), must hold.

\sn
3:) Assume that $\cF$ is a dense subset of $M(F|R)$ and that $\eta \in \cF$. We wish to show that 
$\cF\setminus \{\eta\}$ is still dense in $M(F|R)$. Suppose not. Then there is an nonempty open
subset $V$ of $M(F|R)$ such that $\cF\cap (V\setminus\{\eta\}) = (\cF\setminus\{\eta\})\cap
V=\emptyset$. But by part 3) of Theorem~\ref{topspaces}, $V$ is infinite and hence $V\setminus
\{\eta\}$ is a nonempty open set, so we obtain a contradiction to the density of $\cF$.

\sn
4:) Assume first that $H(F|C)$ admits a minimal representation. Due to part 3) of our theorem we obtain that
$C=R$. We are therefore dealing with the case that $H(F|R)$ admits a minimal representation.
Then necessarily $\trdeg(F|R)=1$; this is proven by Sch\"ulting in \cite[Section 1]{Sch2} for
$R=\R$. Transferring the arguments, one can deduce this from \cite[3.13]{BS} also for the case 
of any real closed base field.

Now consider $H(F|R)$ in the case of $\trdeg(F|R)=1$ and choose an arbitrary $\eta\in M(F|R)$.
We are going to prove that even the family $M(F|R)\setminus \{\eta\}$ does not provide the
intersection $(a)$ for $H(F|R)$. This will imply claim $4)$.

We start by constructing a smooth affine model $X$ on which all places in $M(F|R)$ admit a
center. Assume $F=R(x_1,\dots,x_n)$, then set $x_0=1$ and $y_i=x_i/\sum_{j=0}^n x_j^2$ for 
$i=0,\dots,n$. The elements $y_i$ belong to $H(F|R)$ and generate the function field $F$ over
$R$. Setting $S=R[y_0,\dots,y_n]$ we have found an affine $R$-algebra inside $H(F|R)$
with quotient field $F$. 

Take the integral closure $T$ of $S$ in $F$. We still have that $T\subseteq H(F|R)$ as the
Pr\"ufer ring $H(F|R)$ is integrally closed. Due to the Noether Normalization Theorem we obtain that
$T$ is an affine $R$-algebra, say $T=R[t_1,\dots,t_m]$. We observe that $T$ is a noetherian,
integrally closed ring of dimension $1$; in other words, we have constructed a smooth affine
model $X$ with the Dedekind domain $T\subseteq H(F|R)$ as its coordinate $R$-algebra.

Given any $\lambda \in M(F|R)$, it is finite on $H(F|R)$, so it induces a $R$-epimorphism $\phi:
T\rightarrow R$ the kernel of which is the maximal ideal $\mathfrak{m}$ of a point $a\in X(R),$
the center of $\lambda$. As $T$ is a Dedekind domain, the local ring at the point $a$, i.e.,
$T_\mathfrak{m}$, is a valuation ring; it is also contained in the valuation ring of $\lambda$.
Using again that $T$ is a Dedekind ring, we see that both valuation rings must be equal. Finally, we find
that $\lambda$ is determined by the induced epimorphism $\phi$.

Now consider the place $\eta\in M(F|R)$ chosen above. The elements $s_i=t_i-\eta(t_i), \ i=1,\dots,m$, also generate the algebra $T$. Thus if
$\lambda \neq \eta$, then $\lambda(s_i)\neq \eta(s_i)=0$ for some $i$. It follows that for
$s=\sum_{1\leq i\leq m} s_i^2$, the element $1/s$ lies in the valuation ring of
$\lambda$ but not in the valuation ring of $\eta$. This shows that $\bigcap_{\lambda \in M(F|R)\setminus\{\eta\}}
\cO_\lambda$ is strictly larger than $H(F|R)$.
\end{proof}

\parm
For a function field $F$ over a non-archimedean real closed field $R$ we will give yet another, geometric
representation of $H(F)$ which makes use of of the family of all smooth projective models of $F$. In doing so we base our arguments on Hironaka's celebrated theory of the resolution of singularities in characteristic zero. One may consult Kollar's excellent presentation of this theory, in particular Chapter 3 of his book \cite{Ko}.

Consider any smooth projective model $X$ of $F$. For $x \in X$ we denote by
$\cO_x$ the local ring in $x$ and by $X(R)$ the set of rational points of $X$. For $f\in F$, by $X_f$ we
denote the set of those rational points for which $f \in \cO_x$, i.e., $f$
is defined in $x$. Note that $X_f$ is open in the Zariski topology on $X(R)$ and there exists an open nonempty affine subvariety $Y$ with $f\in \cO(Y)$, implying that $Y(R)\subseteq X_f$. As $X$ is smooth, every point in $X(R)$
is the center of some $\lambda\in M(F|R)$.

Using these facts, Artin's solution to Hilbert's 17th Problem can be rephrased as follows:
\[
f \in \sum F^2 \Leftrightarrow f(x)\geq 0  \text{ for every } x\in X_f\>.
\]
The implication "$\Rightarrow$" can be proven as follows: the point $x\in X_f$ is the center of a place $\lambda \in M(F|R)$ and $f\in \cO_x$, hence $f=\sum_1^k f_i^2$ is contained in the valuation ring $V$ of $\lambda$. Now, this valuation ring is a real valuation ring, which implies that each $f_i\in V$. Therefore, $f(x)=\lambda(f)=\sum_i \lambda(f_i)^2 = \sum_i f_i(x)^2\geq 0$. Concerning the other implication, we note that $f$ is defined on an affine model $Y$ of $F$ and non-negative on $Y(R)$. Artin's theorem then states that $f\in \sum F^2$.

\vspace{.2cm}
The above characterization suggests to look for a geometric characterization of $H(F)$ by appealing to the sets $X_f$ whenever $f\in H(F)$.  

Take a function $f \in H(F)$ and take $x \in X(R)$ such that $f \in \cO_x$.
Since $x$ is the center of some $\lambda\in M(F|R)$, we obtain that
$f(x) = \lambda(f)$. Since $\xi_R\circ \lambda$ is an $\R$-place and $f \in H(F)$, we have that
$\xi_R\circ \lambda(f)\neq \infty$. Therefore, $\lambda(f) \in \cO_R=H(R)$.
We have shown:
\[
f \in H(F)\Rightarrow f(x)\in H(R)\text{ for every } x\in X_f\>.
\]
The converse is in general not true. If $R$ is an archimedean real closed field, then
$H(R) = R$ and the right hand side of the implication is always true, while the left hand side is not.

To understand what is going on, we have to turn again to the relative real holomorphy ring
$H(F|R)$ and its geometrical description given by Sch\"{u}lting in \cite{Sch}:
\begin{equation}                                     \label{geodes}
H(F|R) \>=\> \{ f\in F\mid f \text{ is bounded on } X_f \text{ by elements of } R\}\>.
\end{equation}

For a smooth real projective variety $X$, define
\[
H_X \>:=\> \{f\in F\mid f(x) \in H(R) \text{ for every } x\in X_f\}\>.
\]
Then $H(F) \subseteq H_X$. But functions in $H_X$ are not necessarily bounded on $X_f$ in the case of an
archimedean ordered base field $R$. But in the non-archimedean case every function in $H_X$ is bounded by the
elements with negative values under $v_R$. Therefore, for $R$ non-archimedean we have:
\begin{equation}                            \label{H_X}
H(F) \>\subseteq\> H_X \>\subseteq\> H(F|R)\>,
\end{equation}
where the latter inclusion follows from (\ref{geodes}).
\begin{proposition}                          \label{thhr}
Take a function field $F$ over a non-archimedean real closed field $R$.
Then $H(F)$ is the intersection of the sets $H_X$ where $X$ runs through all smooth projective models of $F$.
\end{proposition}
\begin{proof}
As we observed before, $H(F) \subseteq H_X$ for any smooth model of $F$. Therefore,
$H(F) \subseteq \bigcap \{H_X\mid  X \text{ smooth projective model of } F\}$.

Assume that $f$ is in the intersection of the sets $H_X$. Since $R$ is non-archimedean, (\ref{H_X}) shows that $f$
is in $H(F|R)$. By a theorem of Sch\"{u}lting (see \cite[page 437]{Sch}) there is
a smooth projective model $X_0$ such that $f$ is regular in every point of $X_0(R)$.
Take any $\R$-place $\xi$ such that $\xi = \xi_R \circ \lambda$,  $\lambda \in M(F|R)$.
Since $f \in H(F|R)$, we have $f \in \cO_{\lambda}$. The place $\lambda$ has a center $c(\lambda)$ on the
projective model, so $c(\lambda)\in X_0(R)$.
Then $\lambda(f)=f(c(\lambda)) \in H(F)$ by our assumption. This means that $\xi(f) \neq \infty$, so
$f \in \cO_{\xi}$ for every $\xi \in M_R(F)=\{\xi_R \circ\lambda \mid \lambda\in M(F|R)\}$. Thus $f \in
\bigcap\{\cO_{\xi_R\circ \lambda}\mid\lambda\in M(F|R)\}$, which by (\ref{H(F)}) is equal to $H(F)$.
\end{proof}

\pars
Note that this theorem is not true for an archimedean real closed field $R$ since in this case $H_X=F$ for every
smooth projective model $X$ of $F$.

\mn
%
%
\section{The Real Spectrum of $H(F|R)$ and $H(F)$}       \label{secSper}
%
As before, $F$ denotes a formally real function field over a real closed base field $R$.
\\The topologies on $M(F|R)$ and $M(F)$ find natural interpretations via the theory of the real spectrum
$\Spec_r(A)$ of a commutative ring $A$. Regarding general concepts and results we refer to \cite[Chapter 7]{BCR}
and \cite[Kapitel III]{KS}; however, note that the authors of the latter reference are using the notation $Sper A$
for the real spectrum of $A$. The real spectrum $\Spec_r(A)$ is a quasi-compact space; we
reserve the term ``compact'', in contrast to the use in \cite{BCR}, for quasi-compact Hausdorff spaces.
It is its compact subspace of closed points $\MaxSpec_r(A)$ that we are mainly interested in.

\pars
In our situation, we will prove:
\begin{proposition}      \label{diagram}
There is a commutative diagram
\begin{center}
$
\begin{matrix}
M(F|R) & \stackrel{i}{\longrightarrow} & \MaxSpec_r(H(F|R)) \\
\bigg\downarrow~ \iota_{F|R} &\text{///} &\bigg\downarrow~ \tau\\
M(F) & \stackrel{j}{\longrightarrow} & \MaxSpec_r(H(F))
\end{matrix}
$
\end{center}
where
\begin{enumerate}
\item the maps  $i,\iota_{F|R}$ are topological embeddings with dense images,
\item the map  $j$  is a homeomorphism,
\item the map  $\tau$  is continuous and surjective.
\end{enumerate}
\end{proposition}
Using this proposition and results from real algebraic geometry over arbitrary real closed fields
(see \cite{BCR,DK}), we can prove:
\begin{proposition}                \label{components}
$M(F)$ has only finitely many connected components.
\end{proposition}

It was already known that the space $M(F)$ of a rational function field $F=R(X_1,\dots,X_n)$ is
connected, see \cite[Theorem 2.12]{Ha}.

\pars
Let $A$ denote any commutative ring. By definition, the real spectrum $\Spec_r(A)$, as a set, is the collection of
all \bfind{prime cones} $\alpha \subsetneq A$ satisfying the conditions
\[
\alpha+\alpha\subseteq \alpha,\, \alpha\cdot \alpha \subseteq \alpha,\, \alpha \cup \alpha =A,\,
\alpha\cap -\alpha \text{ is a prime ideal of } A\,.
\]
Let a prime cone $\alpha$ be given. We set $\supp(\alpha)=\alpha\cap-\alpha$. This prime ideal is called the
\bfind{support} of $\alpha$. By the \bfind{residue field of $\alpha$} we will mean the quotient field of
$A/\supp(\alpha)$. Given any $a\in A$ and $\alpha \in \Spec_r(A)$, we write $a(\alpha):=a+\supp(\alpha)$.
An ordering $\bar{\alpha}$ with order relation $\leq_\alpha$ (or in short, $\leq$) is induced by requiring, for all $a\in A$,
\[
0\leq a(\alpha) \>\Leftrightarrow\> a\in \alpha \>
\]
hence $a(\alpha)>0\Leftrightarrow a\in\alpha\wedge -a\notin\alpha$.

The topology on $\Spec_r(A)$ is defined by the following family of basic open sets:
\[
\widetilde{\cU}(a_1,\dots,a_n)=\{\alpha\mid a_1(\alpha)>0,\dots,a_n(\alpha)>0\}
\]
for all $n\in \mathbb{N}, a_1,\dots,a_n \in A$. As ring homomorphisms $\phi: A\rightarrow B$ are well behaved with
respect to the assignment $A\mapsto \Spec_r(A)$, we are dealing with a contravariant functor $\Spec_r$ from the
category of rings to the category of quasi-compact spaces. Here we will only be using the simplest case, where $A$
is a subring of the ring $B$. It is readily seen that we obtain a continuous map, the restriction
\[
\res\>=\> \res_{A,B}:\> \Spec_r(B)\rightarrow \Spec_r(A), \;\;\alpha \mapsto \alpha\cap A\>.
\]
If $\alpha,\beta\in \Spec_r(A)$ satisfy $\alpha\subseteq \beta$, then $\beta$ is called a
\bfind{specialization} of $\alpha$ and $\alpha$ a \bfind{generalization} of $\beta$. The specializations of
a given prime cone $\alpha$ form a totally ordered set with respect to inclusion, and there is a unique maximal
specialization of $\alpha$, denoted by $\rho(\alpha)$. The maximal prime cones are exactly the closed points in
$\Spec_r(A)$. For example, a prime cone whose support is a maximal ideal is a maximal prime cone. We set
\[
\MaxSpec_r(A)\>=\> \{\alpha\in \Spec_r(A)\mid \alpha \text{ maximal}\,\}\>.
\]
It turns out that the subspace $\MaxSpec_r(A)$ is compact and that the specialization map
\[
\rho\>=\> \rho_A:\> \Spec_r(A)\rightarrow \MaxSpec_r(A), \;\; \alpha\mapsto \rho(\alpha)
\]
is continuous and a closed retraction, see \cite[7.1.25]{BCR} and \cite[p.128, Satz 5]{KS}. By composing the
assignment
$A\mapsto \Spec_r(A)$ with the specialization map, we obtain a functor $A\mapsto \MaxSpec_r(A)$ into the category
of compact spaces. In the case where $A$ is a subring of $B$ we obtain the continuous map
\[
\tau=\tau_{A,B}:=\rho_A\circ \res_{A,B}:\> \MaxSpec_r(B)\rightarrow \MaxSpec_r(A)\>.
\]
In what follows we will use the following, easily proven observation: if $\beta$ is a specialization of $\alpha$
and $\supp(\alpha)=\supp(\beta)$, then $\alpha=\beta$. As already stated above, if $\supp(\alpha)$ is a maximal
ideal, then $\alpha$ is a maximal prime cone.

\parb
We now turn to the proof of Proposition~\ref{diagram}.
\begin{proof}
The map $\iota_{F|R}$ has already been introduced and shown in Theorem~\ref{topspaces} to be a topological
embedding of $M(F|R)$ into $M(F)$; by Proposition~\ref{dense}, its image $M_R(F)$ is dense in $M(F)$. The map
$\tau$ on the right hand side equals $\tau_{A,B}$ for $A=H(F),B=H(F|R)$. So it is continuous. Surjectivity follows
once the statements on the maps $i,j$ and the commutativity of the diagram have been shown; this is seen as
follows. As we are
dealing with compact spaces the image of $\tau$ is closed, and furthermore, it contains the image of the dense
subspace $M_R(F)$ under the homeomorphism $j$. All this implies that $\tau$ is surjective.

To define the map $i: M(F|R)\rightarrow \MaxSpec_r(H(F|R))$ and study its properties we need the following facts.
A place $\lambda\in M(F|R)$ induces a $R$-epimorphism $\phi: H(F|R)\rightarrow R$ whose kernel $\mathfrak{p}_\lambda$
is a maximal ideal of $H(F|R)$. As this ring is a Pr\"ufer ring we see that the valuation ring $V$ of $\lambda$
is just the localization $H(F|R)_{\mathfrak{p}_\lambda}, V=\{a/b \mid a,b \in H(F|R), b\notin \mathfrak{p}_\lambda$\}, hence $\lambda(a/b)=\phi(a)/\phi(b)$. Altogether we obtain that the places in $M(F|R)$ are determined by their restriction to $H(F|R)$. Using the unique ordering on $R$ we now define the natural map
\begin{eqnarray*}
i:M(F|R)&\rightarrow& \MaxSpec_r(H(F|R)),\\
\lambda &\mapsto& \alpha_\lambda:=\{a\in H(F|R)\mid \lambda(a)\geq 0\}\>.
\end{eqnarray*}
We observe that $\supp(\alpha_\lambda)=\mathfrak{p}_\lambda$, so indeed, $\alpha_\lambda$ is a maximal prime cone,
as its support is a maximal ideal. As each $\lambda\in M(F|R)$ is the identity on $R$ we find that for each $a\in
H(F|R)$ we have $a-\lambda(a)\in \mathfrak{p}_\lambda$. From this the injectivity of $i$ follows: indeed, if
$\alpha_\lambda=\alpha_\mu\,$, then $\mathfrak{p}_\lambda=\mathfrak{p}_\mu$, so $\mu(a-\lambda(a))=0$ and therefore
$\mu(a)=\lambda(a)$ for every $a\in H(F|R)$, whence $\mu=\lambda$. In addition we obtain that $a(\alpha_\lambda)=
\lambda(a)$ for any $a\in H(F|R)$, from which we deduce:
\[
i^{-1}\left(\widetilde{\cU}(a_1,\dots,a_n)\cap \MaxSpec_r(H(F|R))\right)=V(a_1,\dots,a_n)\>.
\]
This means that the map $i$ is a topological embedding. To prove that the image is dense, consider a nonempty
basic open subset 
$$\overline{\cU}=\widetilde{\cU}(a_1,\dots,a_n)\cap \MaxSpec_r(H(F|R))$$
and pick one of its elements $\alpha$. Set $\mathfrak{p}=\supp(\alpha)$. The residue field of the valuation ring
$H(F|R)_\mathfrak{p}$ equals the residue field of $\alpha$. Therefore we can pull back the ordering $\bar{\alpha}$
to construct an ordering $>$ on $F$ which satisfies $a_i>0$ for $1\leq i\leq n$. Now the arguments presented in
the proof of Proposition~\ref{TopMF|R} yield the existence of $\lambda\in M(F|R)$ with $\lambda(a_i)>0$ for
$1\leq i\leq n$.
We see that $\alpha_\lambda\in \overline{\cU}$.

In the case of the map $j$ we follow a similar route. A place $\xi\in M(F)$ induces a homomorphism
$H(F)\rightarrow \R$. We define
\[
j: M(F)\rightarrow \Spec_r(H(F)), \xi \mapsto \alpha_\xi:=\{a\in H(F)\mid \xi(a)\geq 0\}\>.
\]
This time however, the kernel $\mathfrak{p}_\xi=\supp(\alpha_\xi)$ need not be a maximal ideal.
Nevertheless, $\alpha_\xi\in \MaxSpec_r(H(F))$. To see this, first note that the residue field of $\xi$
equals the residue field of $\alpha_\xi$, which embeds into $\R$. Hence the induced ordering
$\overline{\alpha_\xi}$ is nothing but the pullback of the natural ordering on $\R$. Thus it is an
archimedean ordering of the residue field.

Now assume that $\alpha_\xi\subsetneq \beta$ for some $\beta \in \Spec_r(H(F))$; we wish to deduce a
contradiction. Then, due to the above observation, we obtain that
$\mathfrak{p}:=\supp(\alpha_\xi)\subsetneq \mathfrak{q}:=\supp(\beta)$. Then we can choose $a\in \mathfrak{q}
\setminus \mathfrak{p}$, and we can assume that $a\in \alpha_\xi$ since otherwise, we can replace $a$ by $-a$.
For each rational number $r>0$ we have that $r+a\in \alpha_\xi$
but also $r-a\in \alpha_\xi$: if not, then we would obtain that $r-a \in -\alpha_\xi\subseteq -\beta$ and
$r-a\in \beta$ as $a\in \pm \beta$. This would imply that $r-a\in \mathfrak{q}$, which leads to the contradiction
$r\in \mathfrak{q}$. Passing
to the residue field we see that the non-zero element $\bar{a}$ is infinitesimally small relative to the
archimedean ordering $\overline{\alpha_\xi}$: a contradiction to our assumption. Thus the image of $j$ is
contained in $\MaxSpec_r(H(F))$.

To prove the injectivity of $j$ assume that $\alpha_\xi=\alpha_\zeta$. Then both places have the same valuation
ring and the same residue field on which they induce embeddings $\bar{\xi}, \bar{\zeta}$ into $\R$, subject to
the condition $\bar{\xi}(\bar{a})>0\Leftrightarrow \bar{\zeta}(\bar{a})>0$ for every $a\in H(F)$. As $\Q$
is dense in $\R$ we find that $\bar{\xi}=\bar{\zeta}$, whence $\xi=\zeta$.

From the equivalence $a(\alpha_\xi)>0\Leftrightarrow \xi(a)>0$ we find
that $j$ is a topological embedding of $M(F)$ into $\MaxSpec_r(H(F))$.

Now we show that $j$ is surjective. Consider any $\alpha\in \Spec_r(H(F))$. We want to show that $\alpha\subseteq
\alpha_\xi$ for some $\xi\in M(F)$. This, of course, will settle our claim. Set $\mathfrak{p}=\supp(\alpha)$. Then
$H(F)_\mathfrak{p}$ is the valuation ring of a place $\zeta: F\rightarrow k(\mathfrak{p})\cup
\infty$, where $k(\mathfrak{p})$ is the quotient field of $H(F)/\mathfrak{p}$. It is known that
$H(F)/\mathfrak{p}=H(k(\mathfrak{p}))$, see \cite[1.4]{Sch1}. The ordering $\bar{\alpha}$ induces 
a real place $\lambda_{\bar{\alpha}}$
with a valuation ring which contains $H(k(\mathfrak{p}))=H(F)/\mathfrak{p}$. Using the residue map
$\pi:H(F)\rightarrow H(k(\mathfrak{p}))$, we find $\xi\in M(F)$, determined by the condition
$\xi|_{H(F)}=\lambda_{\bar{\alpha}}\circ \pi$. One readily checks that $\alpha\subseteq \alpha_\xi\,$.

It remains to address the commutativity of the diagram. Starting with $\lambda\in M(F|R)$ we have to show that
\[
\rho\left(\alpha_\lambda\cap H(F)\right)=\alpha_\xi \text{ with }\xi=\xi_R\circ\lambda\>.
\]
As $\alpha_\xi$ is a maximal prime cone it is sufficient to prove that $\alpha_\lambda\cap H(F)\subseteq
\alpha_\xi$. Pick any $a\in H(F)$ with $\lambda(a)\geq 0$. Then $\lambda(a)\in H(R)$ and consequently,
$\xi_R(\lambda(a))\geq 0$, i.e., $a\in \alpha_\xi$.
\end{proof}

Next, the proof of Proposition \ref{components} will be sketched.

\begin{proof}
We know that $\tau$ is continuous and surjective. Therefore, once we know that $\MaxSpec_r(H(F|R))$ has only
finitely many connected components, we can derive the same for $\MaxSpec_r(H(F))$. We list the arguments needed
to show that $\MaxSpec_r(H(F|R))$ decomposes into finitely many connected components. First of all, for any given
ring $A$ the specialization map $\rho: \Spec_r(A)\rightarrow \MaxSpec_r(A)$ induces a bijection between the set of
connected components of $\Spec_r(A)$ and that of $\MaxSpec_r(H(A))$, see for instance \cite[p.129, Satz 6]{KS}.
Consequently, we are facing the problem to show that $\Spec_r(H(F|R))$ admits only finitely many connected
components. This follows from Sch\"ulting's result \cite[p. 436, Theorem]{Sch} as it is known that algebraic sets
over real closed fields decompose into finitely many semi-algebraically connected components. By the way, they
are exactly the semi-algebraic path connected components, see \cite[Sections 2.4.,2.5]{BCR} and \cite[Theorem 4.1]{DK}.
\end{proof}
\n
Note that the surjectivity of $\tau$ can be obtained in a more
direct way by appealing to the Baer-Krull Theorem. But we
preferred to convey the present argument for the sake of a coherent
presentation.

\sn
\begin{remark}
Without providing any further details, we want to conclude by another observation.
The number of connected components $s_F$ of $\Spec_r(H(F|R))$, which is a geometric invariant of $F$, is an
upper bound for the number of connected components $t_F$ of $M(F)$. This is a consequence of the last proof.
However, it may happen that $s_F>t_F\,$, as we will show now.

Take a non-archimedean real closed field $R$, and denote by $R^+$ the set of its positive elements and by $I^+$
the set of its positive infinitesimals.
Take $a\in I^+$. Let $F$ be the function field of the real complete affine curve $C$ given by
\[
 y^2=(x^2-a^2)(1-x^2)\>.
 \]
The relative real holomorphy $H(F|R)$ equals the coordinate ring $A:=R[C]$ and is a Dedekind ring.
The curve $C$ has two semialgebraic connected components separated by the function $x$.

The real spectrum $\Spec_r(A)$ consists of the prime cones
\[
P(\alpha,\beta)\>:=\>\{f\in A\mid f(\alpha,\beta)\geq 0\}
\]
attached to the points $(\alpha,\beta)\in C$ and the prime cones $P\cap A$, where $P$ runs through
the orderings of $F$. The first ones are maximal prime cones. A prime cone of the
second type is maximal if and only if $(F,P)$ is archimedean over $A$, and this holds if and only if $(F,P)$ is
archimedean over $R$ (see \cite[Corollary 5, p.\ 134]{KS}).

Take the ordering
\[
P\>:=\> \{f \mid \exists d \in I^+ \exists e \in R^+ \setminus I^+ : f(c) >0\mbox{ for all } c \in (d,e)\}
\]
of $R(x)$. The ordering $P$ has exactly one extension $P'$ to $F$ in which $y$ is positive.
Take the automorphism $\sigma$ of $F$ such that $\sigma(x) = -x$ and  $\sigma(y) = y$.
Then $Q' = \sigma(P')$ is an ordering of $F$ such that $\lambda_{P'} = \lambda_{Q'}$.
Since $P$ is archimedean over $R$, the same is true for $P'$ and $Q'$.
The function $x$ is positive in $P'$ and negative in $Q'$, therefore  $P' \cap A$ and $Q'\cap A$ belong to
different components of $\Spec_r(H(F|R))$. But the map $\tau$ from Proposition~\ref{diagram} sends
the maximal prime cones $P' \cap A$ and $Q' \cap A$ to prime cones related with the real place $\lambda_{P'}=
\lambda_{Q'}$, which shows that the number of components drops.

The example above was also studied in the paper \cite{KK}, where the relation between cuts on the real curve and
the orderings of its function field was described. In general, the study of $t_F$ and its comparison to $s_F$
seem to be an interesting task.
\end{remark}



\begin{thebibliography} {99}
\bibitem{Be} Becker, E.:{\it Valuations and real places in the theory of formally real fields}, Lect. Notes in Mathematics {\bf 959} (1982), 1-40

\bibitem{[BK]} Blaszczok, A.\ -- Kuhlmann, F.-V.: {\it Algebraic independence of
elements in completions and maximal immediate extensions of valued
fields}, J.~Alg.\ {\bf 425} (2015), 179--214

\bibitem{BCR} Bochnak, J.\ -- Coste, M.\ -- Roy, M.-F.: {\it Real Algebraic Geometry}, Springer Verlag, 1998

\bibitem{Bo} Bourbaki, N.: {\it Commutative algebra}, Paris (1972)

\bibitem{BS} Br\"ocker, L.\ -- Sch\"ulting, H.-W.: {\it Valuations of function fields}, J.\ reine angew.\ Math.\
{\bf 365} (1986), 12--32

\bibitem{BuKu} Buchner, M. A.\ -- Kucharz, W.: {\it On relative real holomorphy rings}, Manuscripta Math.\ {\bf 6}3
(1989), 303--316

\bibitem{C} Craven, T.~C.: {\it The topological space of orderings of a rational function field}, Duke Math.\ J.\
{\bf 41} (1974), 339--347

\bibitem{DK} Delfs, H.\ -- Knebusch, M.: {\it Semialgebraic topology over a real closed field I: Paths and
Components in the Set of Rational Points of an Algebraic Variety}, Math.\ Z.\ {\bf 177} (1981), 107--129

\bibitem{DP} Delzell, C.~N.\ -- Prestel, A.: {\it Mathematical logic and model theory. A brief introduction},
Expanded translation of the 1986 German original. Universitext. Springer, London, 2011

\bibitem{D} Dubois, D.~W.: {\it Infinite primes and ordered fields}, Dissertationes Math.\ {\bf 69} (1970), 1--43

\bibitem{ELW} Elman, R.\ -- Lam, T. Y.\ -- Wadsworth, A. R.: {\it Orderings under field extensions},
J.\ Reine Angew.\ Math.\ {\bf 306} (1979), 7--27

\bibitem{Fo} Fontana, M.\ -- Huckaba, J.A.\ -- Papick, I.J.: {\it Pr\"ufer Domains}, Marcel Dekker, Inc. 1997

\bibitem{Ha} Harman, J.: {\it Chains of higher level orderings}, in: Ordered fields and real algebraic geometry
(San Francisco, Calif., 1981), 141–174, Contemp.\ Math.\ {\bf 8}, Amer.\ Math.\ Soc., Providence, R.I., 1982

\bibitem{[H]} {Hochster, M.$\,$: {\it Prime ideal structure in commutative
rings}, Trans.\ Amer.\ Math.\ Soc.\ {\bf 142} (1969), 43--60}

\bibitem{JR} Jarden, M.\ -- Roquette, P.: {\it The Nullstellensatz over p-adically fields}, J.\ Math.\ Soc.\
Japan {\bf 32} (1980), 425--460

\bibitem{[K--K1]} {Knaf, H.\ -- Kuhlmann, F.-V.$\,$: {\it Abhyankar places
admit local uniformization in any characteristic}, Ann.\ Scient.\ Ec.\
Norm.\ Sup.\ {\bf 38} (2005), 833--846}

\bibitem{KS} Knebusch, M.\ -- Scheiderer, C.: {\it Einf\"uhrung in die reelle Algebra}, Vieweg Verlag 1989

\bibitem{Ko} Kollar, J.: {\it Lectures on Resolution of Singularities}, Annals of Mathematics Studies 166, Princeton University Press 2007

\bibitem{KK} Koprowski, P.\ -- Kuhlmann, K.: {\it Places, cuts and orderings of function fields}, J.\ Algebra
{\bf 468} (2016), 253--274

\bibitem{K1} Kucharz, W.: {\it Invertible ideals in real holomorphy rings}, J.\ reine angew.\ Math.\
{\bf 395} (1989), 171--185

\bibitem{K2} Kucharz, W.: {\it Generating ideals in real holomorphy rings}, J.\ Alg.\ {\bf 144} (1991), 1--7

\bibitem{[K]} Kuhlmann, F.-V.: \emph{On places of algebraic function fields in arbitrary characteristic.}
Advances in Math.\ {\bf 188} (2004), 399--424

\bibitem{[K5]} {Kuhlmann, F.-V.: {\it Value groups, residue fields and bad
places of rational function fields}, Trans.\ Amer.\ Math.\ Soc.\ {\bf 356} (2004), 4559--4600}

\bibitem{[K8]} {Kuhlmann, F.-V.$\,$: {\it Elimination of Ramification I: The
Generalized Stability Theorem}, Trans.\ Amer.\ Math.\ {\bf 362} (2010), 5697--5727}

\bibitem{[K7]} {Kuhlmann, F.-V.: {\it The algebra and model theory of tame valued
fields}, J.\ reine angew.\ Math.\ {\bf 719} (2016), 1--43}

\bibitem{KP} Kuhlmann, F.-V.\ -- Prestel, A.: \emph{On places of algebraic function fields.}
J.\ Reine Angew.\ Math. 353 (1984), 181--195

\bibitem{Lam} Lam, T.~Y.: {\it Orderings, valuations and quadratic
forms},  CBMS Regional Conf.\ Ser.\ Math.\ {\bf 52}. Published for the
Conf.\ Board of the Math.\ Sciences, Washington 1983

\bibitem{RO} {Robinson, A.$\,$: {\it Complete Theories}, Amsterdam (1956)}

\bibitem{Sch} Sch\"{u}lting, H.~W.: {\it Real holomorphy rings in real algebraic geometry},
Real algebraic geometry and quadratic forms (Rennes, 1981),  Lecture Notes in Math., 959, Springer,
Berlin - New York, 1982, 433--442

\bibitem{Sch1} Sch\"ulting, H.-W.: {\it On real places of a field and their real holomorphy rings},
Comm.\ Alg.\ {\bf 10} (1982), 1239--1284

\bibitem{Sch2} Sch\"{u}lting, H.~W.: {\it Prime Divisors on Real Varieties and Valuation Theory},
J.\ Alg.\ {\bf 98} (1986), 499--514
\end{thebibliography}
\end{document}